\documentclass{amsart}
\usepackage{amssymb,amsmath,amsthm,amsfonts}
\usepackage{framed,comment,enumerate}
\usepackage[pdftex]{graphicx}
\usepackage{xcolor, framed}
\def\R {\mathbb{R}}
\def\C {\mathcal{C}}
\def\N {\mathbb{N}}
\def\S {\mathbb{S}}

\def\bv {\mathbf{v}}
\def\bu {\mathbf{u}}

\def\eps{\varepsilon}
\def\dist{{\rm dist}}

\newcommand{\loc}{\mathrm{loc}}
\newcommand{\Lip}{\mathrm{Lip}}

\newcommand{\wc}{\rightharpoonup}
\newcommand{\de}[1] {\mathrm{d} #1}

\newcommand{\pa}{\partial}
\newcommand{\mf}[1]{\mathbf{#1}}

\DeclareMathOperator*{\tsum}{\textstyle{\sum}}

\DeclareMathOperator{\supp}{supp}
\DeclareMathOperator*{\osc}{osc}
\DeclareMathOperator*{\meas}{meas}

\newtheorem{proposition}{Proposition}[section]
\newtheorem{theorem}[proposition]{Theorem}
\newtheorem*{theorem*}{Theorem}
\newtheorem{corollary}[proposition]{Corollary}
\newtheorem{lemma}[proposition]{Lemma}

\theoremstyle{definition}
\newtheorem{definition}[proposition]{Definition}
\newtheorem{remark}[proposition]{Remark}
\numberwithin{equation}{section}

\title[Uniform Lipschitz bounds]{Uniform bounds for strongly competing systems:\\ the optimal Lipschitz case}


\begin{document}

\maketitle

\centerline{Nicola Soave}
 \centerline{Justus Liebig Universit\"at Giessen}
   \centerline{Mathematisches Institut}
   \centerline{Arndtstrasse 2, 35392, Giessen, Germany}
\centerline{email: nicola.soave@math.uni-giessen.de; nicola.soave@gmail.com}   

\medskip

\centerline{Alessandro Zilio}
\centerline{Centre d'analyse et de math\'{e}matique sociales}
\centerline{\'{E}cole des Hautes \'{E}tudes en Sciences Sociales}
\centerline{190-198 Avenue de France, 75244, Paris CEDEX 13}
\centerline{email: alessandro.zilio@polimi.it}

\begin{abstract}
For a class of systems of semi-linear elliptic equations, including
\[
    -\Delta u_i=f_i(x,u_i) - \beta u_i\sum_{j\neq i}a_{ij}u_j^p,\qquad i=1,\dots,k,
\]
for $p=2$ (\emph{variational}-type interaction) or $p = 1$ (\emph{symmetric}-type interaction), we prove that uniform $L^\infty$ boundedness of the solutions implies uniform boundedness of their Lipschitz norm as $\beta \to +\infty$, that is, in the limit of strong competition. This extends known quasi-optimal regularity results and covers the \emph{optimal} case for this class of problems. The proofs rest on monotonicity formulae of Alt-Caffarelli-Friedman and Almgren type in the variational setting, and on the Caffarelli-Jerison-Kenig almost monotonicity formula in the symmetric one.
\end{abstract}

\section{Introduction and main results}

In this paper we are concerned with the optimal uniform regularity of families of solutions to strongly competing systems either of Gross-Pitaevskii type:
\begin{equation}\label{be classical}
\begin{cases}
-\Delta u_{i,\beta} +\lambda_{i,\beta} u_{i,\beta}= \omega_i u_{i,\beta}^3 - \beta u_{i,\beta} \sum_{j \neq i} a_{ij} u_{j,\beta}^2 & \text{in $\Omega$} \\
u_{i,\beta} > 0 & \text{in $\Omega$},
\end{cases}
\end{equation}
or of Lotka-Volterra type:
\begin{equation}\label{lv classical}
\begin{cases}
-\Delta u_{i,\beta} +\lambda_{i} u_{i,\beta} = \omega_i u_{i,\beta}^2 - \beta u_{i,\beta} \sum_{j \neq i} u_{j,\beta} & \text{in $\Omega$} \\
u_{i,\beta} > 0 & \text{in $\Omega$},
\end{cases}
\end{equation}
with $i=1,\dots,k$. In both cases, $\Omega \subset \R^N$ is a domain neither necessarily bounded, nor smooth, $\omega_i \in \R$, and $\beta$ is a positive parameter which has to be thought as tending to $+\infty$. In the previous setting, our main results read as follows.

\begin{theorem}\label{thm: be example}
In dimension $N \le 4$, let us assume that $a_{ij}=a_{ji}$ and $(\lambda_{i,\beta})$ is a bounded sequence. Let $\{\mf{u}_{\beta}\}$ be a family of solutions of \eqref{be classical} uniformly bounded in $L^\infty(\Omega)$. Then for every compact set $\Omega' \subset \subset \Omega$ there exists $M>0$ independent of $\beta$ such that
\[
\|\mf{u}_{\beta}\|_{\Lip(\Omega')}:= \|\mf{u}_{\beta}\|_{L^\infty(\Omega')}+ \|\nabla \mf{u}_{\beta}\|_{L^\infty(\Omega')} \le M.
\]
\end{theorem}

\begin{theorem}\label{thm: lv example}
In any dimension $N \ge 1$, let us assume that $\lambda_{i} \in \R$. Let $\{\mf{u}_\beta\}$ be a family of solutions of \eqref{lv classical} uniformly bounded in $L^\infty(\Omega)$. Then for every compact set $\Omega' \subset \subset \Omega$ there exists $M>0$ independent of $\beta$ such that
\[
\|\mf{u}_{\beta}\|_{\Lip(\Omega')}:= \|\mf{u}_{\beta}\|_{L^\infty(\Omega')}+ \|\nabla \mf{u}_{\beta}\|_{L^\infty(\Omega')} \le M.
\]
\end{theorem}
Here and in the rest of the paper we adopt the vector notation $\mf{u}=(u_1,\dots,u_k)$.

\subsection{Introduction to the problem}

The study of the asymptotic behaviour of singularly perturbed equations and systems of elliptic type is a very broad and active subject of research. In recent years, a lot of interest has been given to systems of equations of competing densities, coming from chemical, biological, physical or purely mathematical applications. Typical examples of such systems can fit under the comprehensive model
\[
    - \Delta u_i = f_i(x,u_i) - \beta g_i (u_1, \dots, u_k) \qquad \text{in $\Omega \subset \R^N$,}
\]
where the functions $g_i$, modelling the interaction between the densities, can assume different shapes according to the underlying phenomena.
\begin{enumerate}[(I)]
    \item For models coming from the physics, typically related to the Gross-Pitaevskii equations (see e.g. \cite{PaWa, PiSt, RuCaFu, Timm}), the coupling between the different densities takes a variational form, as in
        \[
            g_i(u_1, \dots, u_k) = u_i \sum_{j \neq i} a_{ij} u_j^2.
        \]
        Here the matrix $a_{ij}$ is assumed symmetric. This interaction is variational since one can easily see that the functions $g_1, \dots, g_k$ are nothing but the partial derivatives of $G(u_1, \dots, u_k) = \sum_{i, j \neq i} a_{ij} u_i^2 u_j^2$. These models are also of importance in other mathematical problems, such as the approximation of optimal partition problems and of harmonic maps to singular manifolds \cite{caflin, rtt}.
    \item In biological or chemical application as in \cite{Mi, ShKaTe}, the interaction term is, in general, more symmetric, as it is derived from some probabilistic reasonings. In these situations one has, for instance,
        \[
            g_i(u_1, \dots, u_k) = u_i \sum_{j \neq i} u_j.
        \]
        Note that the lack of a variational structure is compensated by the symmetry of the competition.
\end{enumerate}
Great efforts have been directed to the description of a precise asymptotic of the solutions of the previous systems when the competition parameter $\beta$ diverges; with this we mean that the main goals have been:
\begin{enumerate}[(a)]
    \item\label{point reg} to develop a common regularity theory for the solutions of the system, which is independent of the strength of the competition $\beta > 0$;
    \item to investigate under which assumptions one can guarantee convergence of the solutions to some limiting profile;
    \item to study the regularity of the class of limiting profiles, both in terms of the densities and in terms of the emerging free boundary problem;
    \item to give qualitative properties and precise estimates of such convergence.
\end{enumerate}
This paper is mainly devoted to the improvement of the known results concerning the first point, since this serves as foundation for the subsequent ones. Before presenting our contribution, we give a brief review of the existing literature; this serves also as a motivation for our work.

\subsection{Uniform bounds in H\"older spaces}

The limiting behaviour of minimal solutions to variational systems of type \eqref{be classical} when $\beta \to +\infty$ has firstly been studied in \cite{ctvNehari,ctvOptimal} by Conti, Terracini and Verzini in the so-called focusing case $\omega_i>0$: it has been shown that any sequence of minimizers of the energy functional associated to \eqref{be classical} is convergent in $H^1(\Omega)$, as $\beta \to +\infty$, to a limiting profile $\mf{u}_\infty$ whose components have disjoint support, that is $u_{i,\infty} u_{j,\infty} \equiv 0$ a. e. in $\Omega$ for every $i \neq j$. This phenomenon, called \emph{phase separation} or \emph{segregation}, reflects the competitive nature of the considered type of interaction, and has been analysed also in the de-focusing case $\omega_i<0$ in \cite{clll} by Chang et. al. Afterwards, a breakthrough in the comprehension of the regularity issues of the phase separation have been achieved in \cite{caflin}, where for the first time Caffarelli and Lin have shown the $\mathcal{C}^{0,\alpha}$-uniform regularity for families of minimizers. As far as the excited states are concerned, probably the most relevant result available in the literature is the following.

\begin{theorem*}[Noris, Tavares, Terracini, Verzini, \cite{nttv}]
In dimension $N \le 3$, let us assume that $a_{ij}=a_{ji}$ and that $(\lambda_{i,\beta})$ is a bounded sequence. Let $\{\mf{u}_{\beta}\} \subset H_0^1(\Omega)$ be a family of solutions of \eqref{be classical} uniformly bounded in $L^\infty(\Omega)$. Then for every $0<\alpha<1$ there exists $M>0$ independent of $\beta$ such that
\[
\|\mf{u}_{\beta}\|_{\mathcal{C}^{0,\alpha}(\overline{\Omega})} \le M.
\]
Up to a subsequence $\mf{u}_{\beta} \to \mf{u}_\infty$ in $\mathcal{C}^{0,\alpha}(\overline{\Omega})$ and in $H^1(\Omega)$, and $\mf{u}_\infty$ is a segregated configuration, that is $u_{i,\infty} u_{j,\infty} \equiv 0$ in $\Omega$ for every $i \neq j$.
\end{theorem*}
Such a result extends and improves previous ones obtained by Wei and Weth in \cite{ww}, where under the same assumptions the equi-continuity of $\{\mf{u}_\beta\}$ was proved in dimension $N=2$. We mention that in \cite{ww} a wider class of systems is considered, including both \eqref{be classical} and \eqref{lv classical}. 

Also in the symmetric setting phase separation phenomena arise in the limit $\beta \to +\infty$.
\begin{theorem*}[Conti, Terracini, Verzini, \cite{ctv}]
In dimension $N \ge 1$, let us assume that $(\lambda_{i,\beta})$ is a bounded sequence. Let $\{\mf{u}_{\beta}\} \subset H^1(\Omega)$ be a family of solutions of \eqref{lv classical} subjected to the boundary conditions
\[
u_i = \varphi_i \qquad \text{on $\pa \Omega$},
\]
where $\varphi_i$ are positive $\Lip(\pa \Omega)$-functions having disjoint supports. Then for every $0<\alpha<1$ there exists $M>0$ independent of $\beta$ such that
\[
\|\mf{u}_{\beta}\|_{\mathcal{C}^{0,\alpha}(\overline{\Omega})} \le M.
\]
Up to a subsequence $\mf{u}_{\beta} \to \mf{u}_\infty$ in $\mathcal{C}^{0,\alpha}(\overline{\Omega})$ and in $H^1(\Omega)$, and $\mf{u}_\infty$ is a segregated configuration.
\end{theorem*}
We point out that $\{\mf{u}_\beta\}$ is uniformly bounded in $L^\infty(\Omega)$ as a consequence of the maximum principle. We also remark that the results in \cite{ctv} are actually more general, in the sense that the reaction term $\omega_i u_i^2-\lambda_{i} u_i$ can be replaced by a general reaction term of type $f_{i}(x,u_i)$ (independent of $\beta$).

All the aforementioned results  are global, in the sense that the solutions are assumed to be defined on \emph{smooth bounded} domains of $\R^N$ and to satisfy \emph{suitable boundary conditions}, and consequently the uniform estimates which are proved hold in the whole $\overline{\Omega}$.

Concerning the regularity of the limit configurations and of their free-boundary, we mainly refer to \cite{caflin} for the variational setting, to \cite{ckl} for the symmetric one, and in particular to \cite{tt}, which provides a unified approach investigating the regularity of a wide class of segregated vector valued functions, including limiting configurations of both the classes of systems. Altogether, the main result which we want to recall in this setting can be stated as follows.

\begin{theorem*}[Caffarelli et al. \cite{ckl}, Caffarelli and Lin \cite{caflin}, Tavares and Terracini \cite{tt}]
Under the assumptions of the previous theorems, let $\mf{u}_{\beta} \to \mf{u}_\infty$ in $\mathcal{C}^{0,\alpha}(\overline{\Omega})$ and in $H^1(\Omega)$ as $\beta \to +\infty$. Then $\mf{u}_\infty$ is Lipschitz continuous in $\Omega$.
\end{theorem*}

As the limiting profile are not $\mathcal{C}^1$, the uniform Lipschitz regularity is optimal for this class of problems.

\subsection{Uniform bounds in Lipschitz spaces}

As a matter of fact, especially in the variational setting, it is still an open question whether one can deduce uniform bounds in the Lipschitz norm. The aim of this paper is to show that this is the case, in a rather general framework. Some results concerning uniform Lipschitz boundedness are already known in the literature, but in some specific cases.

In \cite[Lemma 2.4]{blwz_phase} Berestycki, Lin, Wei and Zhao deal with the variational system $q=2$ in dimension $N=1$, proving the following. Let $\{(u_\beta, v_\beta)\} \in H^1_0([0,1])$ be solutions of
\begin{equation}\label{1 dim var}
 \begin{cases}
        -u_\beta'' + \lambda_{1,\beta} u_\beta = \omega_{1} u_\beta^3 - \beta u_\beta v_\beta^2 &\text{ in $[0,1]$}\\
        -v_\beta'' + \lambda_{2,\beta} v_\beta = \omega_{2} v_\beta^3 - \beta v_\beta u_\beta^2 &\text{ in $[0,1]$}
    \end{cases}
\end{equation}
with uniformly bounded coefficients $(\lambda_{i,\beta})$. If $0 \leq u_\beta, v_\beta \leq C$, then $u_\beta$ and $v_\beta$ are uniformly bounded in the Lipschitz norm.
The proof of such result heavily rests on the ODE aspects of the one dimensional Hamiltonian system.

In \cite[Theorem 3]{ctv} Conti, Terracini and Verzini deal with the symmetric competition $q = 1$. In the case of \emph{two} components without reaction terms, they proved that if $\{(u_\beta, v_\beta)\} \in H^1(\Omega)$ are non-negative solutions of 
\[    \begin{cases}
        \Delta u_\beta = \beta u_\beta v_\beta &\text{ in $\Omega$}\\
        \Delta v_\beta = \gamma \beta u_\beta v_\beta &\text{ in $\Omega$}\\
        u_\beta = \varphi, v_\beta = \psi &\text{ in $\partial \Omega$}
    \end{cases}
\]
with $\gamma >0$ and traces $\varphi, \psi \in \Lip(\partial \Omega)$, then $\{(u_\beta, v_\beta)\}$ is uniformly bounded in the Lipschitz norm.
With a different method, the result has been generalized to systems with an arbitrary number of components (possibly with suitable reaction terms) in \cite{WaZh}.
We refer the interested reader also to the paper \cite{cafroq}, where it is possible to find some extensions of the previous result (involving different kinds of systems, but always restricted to the case of two components).

We emphasize that the existence of uniform Lipschitz bounds is relevant not only for a pure mathematical flavour. As already observed, it is necessary to obtain rigorous qualitative description of phase separation phenomena. This is clearly the case of \cite{blwz_phase}, where the authors derived a precise decay rate for solutions of \eqref{1 dim var} on the interface $\{u_\beta=v_\beta\}$ in dimension $N=1$, strongly using the uniform Lipschitz boundedness of the solutions themselves (the H\"older bounds would not be sufficient for this purpose).


Our aim is twofold: we shall extend the optimal regularity to general cases and, in the mean time, we shall prove local versions of the regularity estimates, avoiding any assumptions on the boundary behaviour of the solutions. This is in the spirit of the classical elliptic regularity theory, and turns out to be particularly useful in blow-up analysis, when one has to deal with sequences of functions defined on expanding domains and hence the global estimates would not be applicable. We mention that a first step in this second direction can be found in \cite[Theorem 2.6]{Wa}, where the author proves that the main results in \cite{nttv} hold also in a local setting. We refer to \cite{BeTeWaWe, FaSo, SoTe, SoZi, Wa} for several applications which rest upon the local nature of such statement. We refer also to the forthcoming paper \cite{stz} for further extensions, see Remark \ref{rem: x_n to free boundary}.

Finally, we mention that uniform regularity issues have been considered for fully non-linear equations in \cite{qui}, for non-local operators in \cite{tvz1,tvz2,vz}, and in a parabolic setting in \cite{DaWaZh2}.

\subsection{Main results} Concerning the optimal regularity problem, our main results, stated in the greatest possible generality, are the following. 

\begin{theorem}[Case (I)]\label{thm: main general be}
Let $p \ge 1$ and let $\Omega \subset \R^N$ be neither necessarily bounded, nor necessarily smooth, where $N \le 2(1+1/p)$ is a positive integer. Let $\{\mf{u}_\beta\}$ be a family of weak solutions to 
\begin{equation}\label{be completa}
\begin{cases}
-\Delta u_{i,\beta}= f_{i,\beta}(x,u_{i,\beta}) - \beta u_{i,\beta}^p\sum_{j \neq i} a_{ij}  u_{j,\beta}^{p+1}  &\text{ in $\Omega$}\\
u_{i,\beta} > 0 &\text{ in $\Omega$,}
\end{cases}
\end{equation}
with $a_{ij}=a_{ji}$, uniformly bounded in $L^\infty(\Omega)$: $\|\mf{u}_{\beta}\|_{L^\infty(\Omega)} \le m$ for some $m>0$. Let us assume that $\mf{f}_\beta \in \C (\Omega \times \R)$ are such that
    \begin{equation}\label{be assumpt}
        \max_{s \in [0,m]} \sup_{x \in \Omega} \left|\frac{f_{i,\beta}(x,s)}{s}\right| \leq d
    \end{equation}
for some $d>0$; moreover for any sequence $\beta_n \to +\infty$ there exist a subsequence (still denoted $\beta_n$) and a function $\mf{f} \in \mathcal{C}^1(\Omega \times \R)$ such that $\mf{f}_{\beta_n} \to \mf{f}$ in $\mathcal{C}_{\loc}(\Omega \times \R)$. Then for any $\Omega' \subset\subset \Omega$ there exists $M>0$ such that
\[
    \|\mf{u}_\beta\|_{\Lip(\bar \Omega')} = \|\mf{u}_\beta\|_{L^{\infty}(\bar \Omega')}+\|\nabla \mf{u}_\beta\|_{L^\infty(\bar \Omega')} \leq M.
\]
\end{theorem}


As a second result we establish the uniform Lipschitz boundedness of solutions of symmetric systems with an arbitrary number of components and general reaction terms, thus extending the results of \cite{ctv,WaZh}.
\begin{theorem}[Case (II)]\label{thm: main general lv}
In any dimension $N \ge 1$, let $\Omega \subset \R^N$ be neither necessarily bounded, nor necessarily smooth. Let $p_i \ge 1$ for $i=1,\dots,k$, and let $\{\mf{u}_\beta\}$ be a family of weak solutions to
    \begin{equation}\label{lv completa}
\begin{cases}
-\Delta u_{i,\beta}= f_{i,\beta}(x,u_{i,\beta}) - \beta u_{i,\beta}^{p_i} \sum_{j \neq i}  u_{j,\beta}^{p_j}  &\text{ in $\Omega$}\\
u_{i,\beta} > 0 &\text{ in $\Omega$,}
\end{cases}
\end{equation}
uniformly bounded in $L^\infty(\Omega)$. Let us assume that $\mf{f}_{\beta}$ maps bounded sets in bounded sets, uniformly in $\beta$. Then, for any $\Omega' \subset\subset \Omega$ there exists $M>0$ such that
\[
    \|\mf{u}_\beta\|_{\Lip(\bar \Omega')} = \|\mf{u}_\beta\|_{L^{\infty}(\bar \Omega')}+\|\nabla \mf{u}_\beta\|_{L^\infty(\bar \Omega')} \leq M.
\]
\end{theorem}


Theorems \ref{thm: be example} and \ref{thm: lv example} follow straightforwardly.

Some remarks are in order.
\paragraph{\textbf{Remarks}}
1) In the variational case this is the first occurrence in the literature of semi-linear terms depending on the independent variable $x$ (in the symmetric setting, such a dependence has been already considered in \cite{ctv}). This enables us actually to give an equivalent, but apparently more general, formulation of the assumptions allowing the reaction terms $f_{i,\beta}$ to depend upon $u_1,\dots,u_k$ and, if needed, also on $\nabla u_1,\dots, \nabla u_k$ in a uniformly bounded way.
Above all, we stress this in order to point out that the locution \emph{variational} and \emph{symmetric} have to be referred to the type of interaction and not to the system per se.\\
2) The restriction on the dimension in the variational case $\rm(I)$ is mainly a technical assumption, related to the subcriticality of the potentials $u_i^{p+1} u_j^{p+1}$ associated to the interaction terms $u_i^p u_j^{p+1}$.  
We point out that it can be easily dropped if one requires, for instance, that the semi-linear terms satisfy the additional assumption
\[
    f_{i,\beta}(x,s) \leq 0 \qquad \text{for $s \in [0, m]$, $i=1,\dots,k$}.
\]
In \cite{caflin} Caffarelli and Lin considered the variational system
\begin{equation}\label{eqn: caf lin min}
    - \Delta u_i = - \beta u_i \sum_{j \neq i} u_j^2
\end{equation}
(\emph{without} any internal reaction term, $f_{i,\beta} \equiv 0$). In their setting, they proved that minimal solutions to such system are uniformly bounded in $\C^{0,\alpha}(\bar \Omega)$ for any $\alpha \in (0,1)$. The proof of Theorem \ref{thm: main general be} can be slightly modified as indicated in the forthcoming Remark \ref{rem: dimension criticality} to obtain the following result, holding in \emph{any} dimension.
\begin{theorem}
In dimension $N \geq 1$, let $\{\mf{u}_\beta\} \in H^1(\Omega)$ be a family of positive solutions of \eqref{eqn: caf lin min}, uniformly bounded in $L^\infty(\Omega)$. Then for any $\Omega' \subset \subset \Omega$ there exists $M > 0$ such that
\[
    \|\mf{u}_\beta\|_{\Lip(\bar \Omega')}  \leq M.
\]
\end{theorem}
\noindent 3) Concerning the assumptions on the reaction terms in Theorem \ref{thm: main general be}, we firstly stress that the pre-compactness assumption on $\{\mf{f}_{\beta}\}$, although technical at a first glance, is natural and is shared by all the known results in the literature. Indeed, when considering the particular system \eqref{be classical}, it is simply given by the requirement that the sequence $(\lambda_{i,\beta})$ is bounded. Regarding \eqref{be assumpt}, we point out that so far limiting configurations of systems of type \eqref{be completa} have been proved to be Lipschitz continuous in \cite{nttv,tt} using such condition. Without it, the Lipschitz regularity of the limiting profiles is still an open problem. In this perspective, we emphasize that Theorem \ref{thm: main general lv}, not making use of \eqref{be assumpt}, establishes the Lipschitz regularity of all limiting profiles of system \eqref{lv completa} for a wider class of reaction terms with respect to those considered in the literature \cite{tt}. \\
4) In the symmetric setting, the possibility of considering different exponents $p_i$ can be used to obtain uniform bounds in more general models. Indeed, with the change of variable $v_i : = u_i^{p_i}$, equation \eqref{lv completa} reads as
    \[
        - \Delta v_i^{1/p_i} = f_i(x, v_1, \dots, v_k) -\beta v_i \sum_{j \neq i} v_j \qquad \text{in $\Omega$}
    \]
that is, the Lotka-Volterra system for the fast-diffusion equation.\\
5) Concerning uniform regularity up to the boundary, we believe that all our results can be extended with some efforts to deal with systems of equations with elliptic operators with variable coefficients. Since we will make use of several monotonicity formulae, this would be very technical and not always easy; the reader can easily understand what we mean by looking at the contribution \cite{matpet}, where the Caffarelli-Jerison-Kenig monotonicity formula has been extended to systems with variable coefficient. We point out that such contribution allows to slightly modify the proof of Theorem \ref{thm: main general lv} to obtain uniform estimates up to the boundary, and then to obtain also the global regularity for solutions of systems with boundary conditions \emph{possibly depending on $\beta$}.
\begin{theorem}\label{thm: lv local}
Under the assumptions of Theorem \ref{thm: main general lv}, let $\Omega$ be smooth and bounded, and let $\mf{u}_{\beta}$ satisfy the boundary conditions
\[
u_{i,\beta} = \varphi_{i,\beta} \qquad \text{on $\pa \Omega$}
\]
in a weak sense, where $\varphi_{i,\beta} \in \Lip(\pa \Omega)$ are uniformly bounded in $\Lip(\pa \Omega)$. Then there exists $M>0$ independent of $\beta$ such that
\[
\|\mf{u}_\beta\|_{\Lip(\overline{\Omega})} \le M.
\]
\end{theorem}
In the proof of Theorem \ref{thm: main general be} we shall make use of new monotonicity formulae which are still not available for operators with variable coefficients. Therefore, for the moment the issue of the uniform regularity up to the boundary remains open, and will be investigated in future works.\\
6) Finally, we point out that local estimates as the ones in Theorem \ref{thm: main general be} and \ref{thm: main general lv} can be used directly to obtain global uniform bounds in the case in which $\Omega = \R^N$. As an example, we have
\begin{theorem}
Let $\{\mf{u}_\beta\}$  be a family of $H^1_\loc(\R^N)$ functions such that $\|\mf{u}_\beta\|_{L^\infty(\R^N)} \leq m$ for some $m > 0$. Let us also assume that $\{\mf{u}_\beta\}$ solves either \eqref{be completa} or \eqref{lv completa} in $\R^N$, under the respective assumptions. Then there exists $M > 0$ independent of $\beta$ such that
\[
    \|\mf{u}_\beta\|_{\Lip(\R^N)} \le M.
\]
\end{theorem}

\medskip

In the proofs of the main results, only for the sake of simplicity, we focus on the particular cases
\begin{equation}\label{be syst}
\begin{cases}
-\Delta u_i= f_{i,\beta}(x,u_i) - \beta u_i \sum_{j \neq i} a_{ij}  u_j^2  &\text{ in $\Omega$}\\
u_i > 0 &\text{ in $\Omega$}
\end{cases}
\end{equation}
and
\begin{equation}\label{lv syst}
\begin{cases}
-\Delta u_i= f_{i,\beta}(x,u_i) - \beta u_i \sum_{j \neq i}  u_j  &\text{ in $\Omega$}\\
u_i > 0 &\text{ in $\Omega$.}
\end{cases}
\end{equation}
The reader can easily check that all the results hold in the generality specified by Theorems \ref{thm: main general be} and \ref{thm: main general lv}, with the same proofs. Moreover, in the proof of Theorem \ref{thm: main general be} we assume that $\Omega \subset \R^N$ with $N \ge 3$. The case $N=2$ is in general easier to deal with, and it can be recovered extending the solution in one further dimension in a constant way.

\subsection{Outline of the proofs (H\"older bounds vs. Lipschitz bounds)}

Here we give a rough idea of the proofs of the main theorems, which we think can serve as a guide towards the rest of the paper. Both the proofs of Theorem \ref{thm: main general be} and \ref{thm: main general lv} proceed by contradiction and are based upon a blow-up analysis. In the following we focus on the variational setting, and we consider the case $f_{i,\beta} \equiv 0$ to simplify the notation. For any compact set $K \subset \subset \Omega$, we aim at showing that the Lipschitz semi-norm of $\mf{u}_\beta$ is bounded in $K$, uniformly in $\beta$. To this aim, we introduce a cut-off function $0 \le \eta \le 1$ such that $\eta \equiv 1$ in $K$ and $\supp \eta=:K' \subset \subset \Omega$. If we prove that for some constant $C>0$ independent of $n$
\[
\sup_{x \in \Omega} |\nabla (\eta u_{i,\beta})| \le C \qquad i = 1, \dots, k,
\]
then the desired result follows. Hence, we assume by contradiction that for a sequence $\beta_n \to +\infty$
\[
L_n:= \sup_i \sup_{x \in \Omega} |\nabla (\eta u_{i,\beta_n})| \to +\infty ;
\]
up to relabelling and up to a subsequence $L_n:=|\nabla u_{1,\beta_n}(x_n)|$ for some $x_n \in K'$. We introduce two blow-up sequences
\[
  v_{i,n}(x) := \eta(x_n) \frac{u_{i,\beta_n}(x_n + r_n x)}{L_n r_n} \quad \text{and} \quad \bar{v}_{i,n}(x) := \frac{(\eta u_{i,\beta_n})(x_n + r_n x)}{L_n r_n},
  \]
where $r_n \to 0$ is chosen in such a way that $\sum_i \bar v_{i,n}(0) = 1$. It is possible to check that both of them are defined in scaled domains exhausting $\R^N$. We point out that $\mf{v}_n$ satisfies an equation similar to that for $\mf{u}_{\beta_n}$ but, at a first glance, does not exhibit any property of compactness. On the other hand, it is not difficult to check that $\bar{\mf{v}}_n$ is uniformly convergent on compact sets to a limit function $\mf{v}$, but does not satisfy any reasonable equation. We shall prove that for every $r>0$
\[
\lim_{n \to \infty} \|\mf{v}_n -\bar{\mf{v}}_n\|_{L^\infty(B_r)} = 0,
\]
so that the convergence $\mf{v}_n \to \mf{v}$ in $\mathcal{C}_{\loc}(\R^N)$ will follows by the convergence of $\bar{\mf{v}}_n$. We show then that the limit function $\mf{v}$ is non-constant and globally Lipschitz continuous in $\R^N$, and has only two non-trivial components, say $v_1$ and $v_2$; moreover, $(v_1,v_2)$ is non-negative and solves either the regular problem
\begin{equation}\label{pb reg intro}
\begin{cases}
-\Delta v_1= -  v_1 v_2^2 & \text{in $\R^N$} \\
-\Delta v_2 = - v_1^2 v_2 & \text{in $\R^N$}
\end{cases}
\end{equation}
or the segregated one
\begin{equation}\label{pb segr intro}
\begin{cases}
-\Delta v_1= 0 & \text{in $\{v_1>0\}$} \\
-\Delta v_2 = 0 & \text{in $\{v_2>0\}$} \\
v_1 \cdot v_2 \equiv 0 & \text{in $\R^N$} \\
-\Delta (v_1-v_2)=0 & \text{in $\R^N$}.
\end{cases}
\end{equation}
The fact that $\mf{v}$ solves either a regular problem or a segregated one depends on the asymptotic relation of the sequences $(r_n)$ and $(L_n)$, which a priori is unknown. In both cases, a relevant fact which marks a striking difference with the present literature concerning uniform bounds in H\"older spaces is represented by the existence of globally Lipschitz continuous solutions for both the previous problems (in particular, in the second one the reader may simply consider $v=x_1^+$, $v=x_1^-$). On the contrary, as proved in \cite{nttv}, globally $\alpha$-H\"older continuous solutions does not exist for any $0<\alpha<1$. This means that in order to reach a contradiction we are not allowed to pass to the limit, but we have to argue directly on the blow-up sequence $\{\mf{v}_n\}$ and to prove a kind of \emph{approximate Liouville-type result}, saying that in the previous setting, the sequence $\{\mf{v}_n\}$ cannot converge to a non-constant globally Lipschitz continuous limiting profile which solves \eqref{pb reg intro} or \eqref{pb segr intro}. We will reach such a result by using Almgren type and Alt-Caffarelli-Friedman type monotonicity formulae in the variational setting, while in the symmetric one we make use of the celebrated Caffarelli-Jerison-Kenig monotonicity formula.

\subsection{Plan of the paper}
Section \ref{sec: asymptotic} concerns the blow-up analysis, which will be considered simultaneously for the variational case and for the symmetric one. In Section \ref{sec: monotonicity} we introduce the monotonicity formulae which will serves as main tools in the proof of Theorem \ref{thm: main general be}; such proof will be the object of Section \ref{sec: variational}. In Section \ref{sec: symmetric} we prove Theorem \ref{thm: main general lv}. We point out that, although all the monotonicity formulae will be applied either to  the sequence $\{\mf{u}_{\beta_n}\}$, or to the blow-up sequence $\{\mf{v}_n\}$, in Section \ref{sec: monotonicity} we will state and prove them in higher generality, in order to provide the reader with results as flexible as possible.

\medskip

\paragraph{\textbf{Some notations}} As usual, $B_r(x_0)$ denotes the open ball of centre $x_0$ and radius $r$. When $x_0=0$, we write simply $B_r$ instead of $B_r(0)$ for the sake of simplicity. The normal derivative and the tangential gradient of a funtion $u$ on a given surface are denoted by $\pa_{\nu}$ and $\nabla_\theta$ respectively. The capital letter $C$ stays for a positive constant which can differ from line to line.

\section{Asymptotic of the blow up sequence}\label{sec: asymptotic}

In this section we consider a system of type
\begin{equation}\label{syst q}
\begin{cases}
-\Delta u_i= f_i(x,u_i) - \beta \sum_{j \neq i} a_{ij} u_i u_j^q  &\text{ in $\Omega$}\\
u_i > 0 &\text{ in $\Omega$,}
\end{cases}
\end{equation}
and we address simultaneously the cases $q=2$ (variational interaction) and $q=1$ (Lotka-Volterra type interaction); in the latter situation, as specified in Theorem \ref{thm: main general lv} we assume that $a_{ij}=1$ for every $i \neq j$. Without loss of generality, we suppose that $\Omega \supset B_3$, and we aim at proving the uniform Lipschitz bound in $B_1$. As in \cite{ctv,nttv, stz, tvz1}, the problem of the uniform bound is tackled with the introduction of suitable blow-up sequences. Let $0 \le \eta \le 1$ be a smooth cut-off function such that $\eta =1$ in $B_1$ and $\eta=0$ in $\R^N \setminus B_{2}$. By definition, the family $\{\eta \bu_\beta\}$ admits a uniform bound on the Lipschitz modulus of continuity if there exists a constant $C > 0$ such that
\begin{equation}\label{lip scaling}
\sup_{i=1,\dots,k} \sup_{\substack{x \neq y \\ x,y \in \overline{B_2}}}    \frac{ |(\eta u_{i,\beta})(x)-(\eta u_{i,\beta})(y)|}{|x-y|} \leq C.
\end{equation}
Since $\eta=1$ in $B_1$, this is sufficient to give the desired result. We first observe that, if $\beta$ is bounded, then a uniform bound of such kind does exist as a consequence of the regularity theory for elliptic equation  (for which we refer, here and in the rest of the paper, to \cite{GiTr}): indeed the right hand side of \eqref{syst q} is in this case uniformly bounded in $L^{\infty}$, and the solution $\bu_{\beta}$ is uniformly $\C^{1,\alpha}$-regular in the interior of $\Omega$, for every $\alpha < 1$. Hence, we only need to consider the case $\beta \to +\infty$. We shall show that there exists $C>0$ such that
\[
\sup_{i = 1, \dots, k} \sup_{x \in \overline{B_2} } |\nabla (\eta u_{i,\beta})(x)| \leq C \qquad \text{for every $\beta \gg 1$.}
\]
Let us assume by contradiction that this is not true and, consequently, that there exists a sequence $\beta_n \to +\infty$ and a corresponding sequence $\{\bu_{\beta_n}\}$ such that
\begin{equation}\label{absurd assumption}
    L_n := \sup_{i = 1, \dots, k} \sup_{x \in \overline{B_2} } |\nabla (\eta u_{i,\beta_n})(x)| \to \infty \qquad \text{as $n \to +\infty$.}
\end{equation}
Up to a relabelling, we may assume that the supremum is achieved for $i = 1$ and at a point $x_n \in B_{2}$. Moreover, in the variational setting $q=2$, thanks to the local version of the main results in \cite{nttv} (which have been proved in absence of the terms $f_{i,\beta_n}$ in Theorem 2.6 of \cite{Wa}, and which will appear in a more general setting in \cite{stz}), we may also choose a subsequence $\{\mf{u}_{\beta_n}\}$ which converges to some limiting profile $\mf{u}$ in $H^1(B_2)$ and in $\C^{0,\alpha}(B_2)$ for every $0<\alpha < 1$. The contradiction argument is based upon two blow-up sequences:
\begin{equation}\label{def blow-up}
    v_{i,n}(x) := \eta(x_n) \frac{u_{i,\beta_n}(x_n + r_n x)}{L_n r_n} \quad \text{and} \quad \bar{v}_{i,n}(x) := \frac{(\eta u_{i,\beta_n})(x_n + r_n x)}{L_n r_n},
\end{equation}
both defined on the scaled domain $(\Omega - x_n)/r_n \supset (B_3-x_n)/r_n=: \Omega_n$. The functions $\bar{\mf{v}}_n$ are non-trivial in the subset $(B_{2}-x_n)/r_n=:\Omega_n'$. We choose the scaling factor $r_n > 0$ in such a way that
\[
    \sum_{i=1}^k \bar{v}_{i,n}(0) = \sum_{i=1}^k\frac{(\eta u_{i,\beta_n})(x_n)}{L_n r_n} = 1 \quad \implies \quad r_n = \sum_{i=1}^k\frac{(\eta u_{i,\beta_n})(x_n)}{L_n} \to 0,
\]
where the last conclusion follows by the uniform $L^\infty$ boundedness of the family $\{\mathbf{u}_{\beta}\}$. The following lemma focuses on some preliminary properties of the blow up sequences.
At first, we define
\[
    f_{i,n}(x,t) := r_n  \frac{\eta(x_n)}{ L_n } f_{i,\beta_n}\left(x_n + r_n x,  t \frac{L_n r_n}{\eta(x_n)}\right).
\]
\begin{lemma}\label{lem: mega lemma}
In the previous blow-up setting, the following assertions hold:
\begin{enumerate}
    \item $f_{i,n}(x,v_{i,n}(x)) \to 0$ uniformly in all $\Omega_n$ as $n \to \infty$;
    \item the scaled domains $\Omega_n$ and $\Omega_n'$ exhaust $\R^N$, that is, $\Omega_n, \Omega_n' \to \R^{N}$ as $n \to \infty$; moreover, $\Omega_n \supset B_{1/r_n}$ for every $n$;
    \item the sequence $\{\bv_{n}\}$ satisfies
    \[
        - \Delta v_{i,n} = f_{i,n}(x,v_{i,n}) - M_n v_{i,n} \tsum_{j \neq i}a_{ij} v_{j,n}^q \quad \text{ in $\Omega_n$},
    \]
    where
    \[
        M_n := \beta_n \left(\frac{L_n}{\eta(x_n)}\right)^{q} r_n^{2+q};
    \]
    \item\label{itm: equihold} the sequence $\{\bar \bv_n\}$ has uniformly bounded $Lip$-seminorm:
     \[
        \sup_{i=1,\dots,k} \sup_{x \neq y} \frac{|\bar v_{i,n}(x) - \bar v_{i,n}(y)|}{|x-y|} \le 1 ;
     \]
furthermore $|\nabla \bar v_{1,n}(0)| = 1$, and $|\nabla v_{1,n}(0)| \to 1$ as $n \to \infty$;
    \item\label{itm: close_seqs} there exists $\bv$, globally Lipschitz continuous in $\R^N$ with Lipschitz constant equal to $1$, such that up to a subsequence both $\bv_n \to \bv$ and $\bar{\bv}_n \to \bv$ in $\mathcal{C}_{\loc}(\R^N)$ as $n \to \infty$;
\item There holds $\bv_n \to \bv$ in $H^1_{\loc}(\R^N)$ as $n \to \infty$, and for any $r>0$ there exists $C>0$ such that
\begin{equation}\label{bound interaction}
\int_{B_r} M_n v_{i,n} \sum_{j \neq i} a_{ij} v_{j,n}^q \le C \qquad \text{for every $i$}.
\end{equation}
If $M_n \to +\infty$, then $v_{i,n} v_{j,n} \to 0$ as $n \to \infty$ for any $i \neq j$.
\end{enumerate}
\end{lemma}
\begin{proof}
Points (1) and (3) are straightforward consequences of the definitions and of our assumptions. \\
(2) Since $0 \in \Omega_n'$ for every $n$, to prove that $\Omega_n' \to \R^N$ it is sufficient to check that $\dist(0,\pa \Omega_n') \to +\infty$ as $n \to \infty$. Firstly we observe that
\[
r_n = \sum_{i=1}^k\frac{(\eta u_{i,\beta_n})(x_n)}{L_n} \le \frac{\|\bu_{\beta_n}\|_{L^\infty(B_2)}}{L_n} \eta(x_n) \le \frac{m l}{L_n} \dist(x_n,\pa B_2),
\]
where $l$ denotes the Lipschitz constant of $\eta$. Therefore,
\[
\dist(0,\pa \Omega_n')  = \frac{\dist(x_n,\pa B_2)}{r_n} \ge \frac{L_n}{m l} \to +\infty \quad \text{as $n \to \infty$}.
\]
The fact that $\Omega_n \supset B_{1/r_n}$ follows by definition.\\
(4) The uniform bound on the Lipschitz seminorm of $\bar{\mf{v}}_n$, and the fact that $|\nabla \bar v_{1,n}(0)|=1$, are direct consequences of the definitions. Moreover
\[
\nabla \bar v_{1,n}(0) = \frac{u_{1,\beta_n}(x_n) \nabla \eta(x_n)}{L_n} + \frac{\eta(x_n) \nabla u_{1,\beta_n}(x_n)}{L_n} = o(1) + \nabla v_{1,n} (0)
\]
as $n \to \infty$.\\
(5) Let $r>0$. The sequence $\{\bar{\bv}_n\}$ has a uniformly bounded Lipschitz seminorm in $\overline{B_r}$, and is uniformly bounded in $0$. Hence, by the Ascoli-Arzel\`a theorem, it is uniformly convergent (up to a subsequence) to some $\bv \in \mathcal{C}(\overline{B_r})$ having Lipschitz-seminorm bounded by $1$. To complete the proof, we show that $\bv_n-\bar{\bv}_n \to 0$ as $n \to \infty$ in $\mathcal{C}_{\loc}(\R^N)$. To this aim, it is sufficient to observe that for any compact $K \subset \R^N$
\[
\sup_{x \in K} |v_{i,n}(x)-\bar{v}_{i,n}(x)| = \sup_{x \in K}\frac{u_{i,\beta_n}(x_n+r_n x)}{L_n r_n} |\eta(x_n)-\eta(x_n+r_nx)| \le \sup_{x \in K}\frac{l m }{L_n} |x|,
\]
where we used the uniform boundedness of $\{\bu_n\}$, and we recall that $l$ denotes the Lipschitz constant of $\eta$. Since $L_n \to +\infty$ and $K$ is compact, the desired result follows. \\
(6) As far as the estimate \eqref{bound interaction} is concerned, it is sufficient to test the equation for $v_{i,n}$ against a smooth cut-off function $0 \le \varphi \le 1$ such that $\varphi =1$ in $B_r$ and $\varphi=0$ in $\R^N \setminus B_{2r}$: we obtain
\[
\int_{B_r} M_n v_{i,n} \sum_{j\neq i} a_{ij} v_{j,n}^q \le \int_{B_{2r}} |f_{i,n}(x,v_{i,n}) \varphi + v_{i,n} \Delta \varphi| \le C,
\]
where we used the boundedness of $\{\mf{v}_n\}$ in compact sets. Testing the equation for $v_{i,n}$ against $v_{i,n} \varphi^2$, we also deduce that
\[
\frac{1}{2}\int_{B_r} |\nabla v_{i,n}|^2 \le 2\int_{B_{2r}} |\nabla \varphi|^2 v_{i,n}^2 + \int_{B_{2r}} \left(f_{i,n}(x,v_{i,n}) v_{i,n} \varphi^2 - M_n v_{i,n}^2 \sum_{j \neq i} a_{ij} v_{j,n}^q \varphi^2\right) \le C,
\]
where, as before, we used the boundedness of $\{\mf{v}_n\}$ on compact sets and the \eqref{bound interaction}. This implies that up to a subsequence $v_{i,n} \wc v_i$ weakly in $H^1(B_r)$. In order to pass from the weak convergence to the strong one, we observe that since $\|v_{i,n}\|_{H^1(B_r)} \le C$ independently of $n$, by replacing if necessary $r$ with a slightly smaller quantity we have also
\[
\int_{\pa B_r} |\nabla v_{i,n}|^2 \le C
\]
independently of $n$. Therefore, by testing the equation for $v_{i,n}$ against $(v_{i,n}-v_i)$ in $B_r$, we deduce that
\begin{align*}
\left|\int_{B_{r}} \nabla v_{i,n} \cdot \nabla(v_{i,n}-v_i) \right| &= \left|\int_{\pa B_r} \pa_{\nu} v_{i,n} (v_{i,n}-v_i)\right|\\
& \qquad +\left|  \int_{B_{r}}  f_{i,n}(x,v_{i,n}) (v_{i,n}-v_i) - M_n v_{i,n} \sum_{j \neq i} a_{ij} v_{j,n}^q (v_{i,n}-v_i) \right|\\
& \le \left( \int_{\pa B_r} |\pa_{\nu} v_{i,n}| + \int_{B_r} \left|f_{i,n}(x,v_{i,n})\right|  + M_n v_{i,n} \sum_{j \neq i} a_{ij} v_{j,n}^q \right) \|v_{i,n}-v_i\|_{L^\infty(B_r)}.
\end{align*}
Recalling that $v_{i,n} \to v_i$ uniformly in $B_{r}$ and that all the other terms are bounded, the desired result follows.
\end{proof}

%

In the rest of this section, we aim at proving that the limit function $\mf{v}$ is non-constant and has exactly two non-trivial components. We have to distinguish between two cases, according to whether $(M_n)$ is bounded or not.
In the former case, the function $\mf{v}$ will be shown to be non-constant as a result of the regularity theory for elliptic equations. In the latter one, the situation is more involved, and we shall make use of the following decay estimate, which allows to treat also more general interaction terms of type $u_i^p u_j^{q}$ with $p,q \ge 1$ falling under the assumptions of Theorems \ref{thm: main general be} or \ref{thm: main general lv}.

\begin{lemma}\label{lem: polynomial decay}
Let $x_0\in \R^N$ and $A, M, \delta,\rho>0$. Let $u \in H^1(B_{2\rho}(x_0)) \cap \mathcal{C}(\overline{B_{2\rho}(x_0)})$ be a subsolution to
\begin{equation}\label{pb confronto}
\begin{cases}
-\Delta u \le -M u^p + \delta & \text{in $B_{2\rho}(x_0)$} \\
u \le A & \text{in $B_{2\rho}(x_0)$}
\end{cases}
\end{equation}
for some $p \ge 1$. Then there exists $C>0$, depending only on the dimension $N$, such that
\[
M u^p(x) \le \frac{CA}{\rho^2}  +  \delta \qquad \text{for every $x \in B_{\rho}(x_0)$}.
\]
\end{lemma}
\begin{proof}
Let $v \in H^1(B_{2\rho}(x_0))$ be a positive solution to
\[
    \begin{cases}
        - \Delta v + M|v|^{p-1}v = 0 &\text{ in $B_{2 \rho}(x_0)$}\\
        v = A &\text{ on $\partial B_{2\rho}(x_0)$}.
    \end{cases}
\]
The existence of such function for any value of $M>0$ and $p\geq1$ can be shown by the direct method of the calculus of variations. Moreover, the weak maximum principle implies that $v \le A$ in $B_{2\rho}(x_0)$. Let $\eta \in \C^{\infty}_0(B_{2\rho}(x_0))$ be a smooth cut-off function such that $\eta = 1$ in $B_{3\rho/2}(x_0)$, $0 \le \eta \le 1$ and $|\Delta \eta| \le C/\rho^2$. Testing the equation for $v$ against $\eta$ we obtain
\[
    \int_{B_{3\rho/2}(x_0)} M v^p  \leq \int_{B_{2\rho}(x_0)} M v^p \eta =  \int_{B_{2\rho}(x_0)} \Delta v \eta =  \int_{B_{2\rho}(x_0)} v \Delta \eta \leq C A\rho^{N-2}.
\]
Let $y \in B_{\rho}(x_0)$. Since $v$ is subhamornic and $p\geq 1$, the mean value theorem gives
\[
    M v(y)^p \leq M \left( \frac{1}{|B_{\rho/2}(y)|}\int_{B_{\rho/2}(y)} v \right)^p \leq \frac{1}{|B_{\rho/2}(y)|} \int_{B_{\rho/2}(y)} M v^p \leq \frac{CA}{\rho^2}.
\]
Let us now consider the auxiliary function $\bar v := v + (\delta/M)^{1/p}$. Trivially, one has $\bar v^p \geq v^p + \delta/M$, and thus
\[
    \begin{cases}
        - \Delta \bar v \geq  - M\bar v^{p} +\delta &\text{ in $B_{2\rho}(x_0)$}\\
        \bar v \geq A  &\text{ on $\partial B_{2\rho}(x_0)$.}
    \end{cases}
\]
Hence, $\bar v$ is a supersolution to \eqref{pb confronto} and the thesis follows applying the comparison principle.
\end{proof}

\begin{lemma}\label{lem: non trivial limit}
The limit function $\bv$ is not constant. In particular, at least the first component $v_1$ is neither trivial nor constant.
\end{lemma}
\begin{proof}
As announced, we divide the proof according to properties of $(M_n)$.

\paragraph{\textbf{Case 1)} \emph{$(M_n)$ is bounded.}}
Since $\{\mf{v}_n\}$ is uniformly bounded in any compact set of $\R^N$, the sequence $\{\Delta v_{1,n}\}$ is uniformly bounded as well; by standard regularity theory for elliptic equations, we deduce that for every compact $K \subset \R^N$ there exists $C>0$ independent of $n$ such that $\|v_{1,n}\|_{\mathcal{C}^{1,\alpha}(K)} \le C$. This implies that, up to a subsequence, the convergence of $v_{1,n}$ to $v_1$ takes place in $\mathcal{C}^{1,\alpha}_{\loc}(\R^N)$ for any $0<\alpha<1$, so that in particular $|\nabla v_{1}(0)| = \lim_n |\nabla v_{1,n}(0)|=1$, and $\mf{v}$ cannot be a vector of constant functions.

\paragraph{\textbf{Case 2)} \emph{$M_n \to +\infty$.}}
By the uniform bound \eqref{bound interaction} we infer that the limiting profile $\mf{v}$ is segregated: $v_i v_j \equiv 0$ in $\R^N$ for every $i \neq j$. Therefore, recalling that by the choice of $r_n$ we have $\sum_{i=1}^k v_{i,n}(0)=1$, there are two possibilities: either $v_1(0)=0$, or $v_1(0) = 1$.

Assume at first that $v_1(0) = 0$. Then there exists $h \neq 1$ such that $v_h(0) =1$, and by continuity of $\mf{v}$ it results that $v_1 \equiv 0$ in an open neighbourhood of $0$. Moreover, $v_{h,n}(0) \ge 7/8$ for every $n$ sufficiently large.
Thanks to points \eqref{itm: equihold} and \eqref{itm: close_seqs} of Lemma \ref{lem: mega lemma}, we have
\[
|v_{h,n}(x) - v_{h,n}(0)| \le  |v_{h,n}(x) - \bar v_{h,n}(x)| + |\bar v_{h,n}(x) - \bar v_{h,n}(0)| \le o(1)+ |x| \le o(1) +\frac{1}{2}
\]
as $n \to \infty$, for every $x \in B_{1/2}(0)$. Thus, whenever $n$ is sufficiently large, $v_{h,n} \ge 1/8$ in $B_{1/2}$.
As a consequence, the equation for $v_{1,n}$ gives
\[
\begin{cases}
-\Delta v_{1,n} \le - C M_n v_{1,n} + \delta & \text{in $B_{1/2}$} \\
v_{1,n} \ge 0  & \text{in $B_{1/2}$} \\
v_{1,n} \le A & \text{in $B_{1/2}$},
\end{cases}
\]
where $\delta \ge \sup_{B_{1/2}(0)} f_{i,n}$ can be chosen independently of $n$, and the upper bound on $v_{1,n}$ in $ B_{1/2}$ follows by the uniform boundedness of $\{\mf{v}_{n}\}$ in compact sets. By Lemma \ref{lem: polynomial decay}, we infer that $M_n v_{1,n} \le C$ in $B_{1/4}$ independently of $n$. Therefore $|\Delta v_{1,n}(x)|  \le C$ for every $x \in B_{1/4}$,
%
%
%
%
which implies that up to a subsequence $v_{1,n} \to v_1$ in $\mathcal{C}^1(B_{1/4})$. In particular $|\nabla v_1(0)|=1$, in contradiction with the fact that $v_1 \equiv 0$ in a neighbourhood of $0$. Thus, if $(M_n)$ is unbounded necessarily $v_{1}(0)=1$, and as a consequence the same argument described above provides $M_n v_{j,n} \le C$ for every $x \in B_{1/4}$ and $j \neq 1$. Using again the uniform boundedness of the sequence $\{\mf{v}_n\}$ in $B_{1/4}$, we infer that $|\Delta v_{1,n}(x)| \le C$ in $B_{1/4}$, and hence
up to a subsequence $v_{1,n} \to v_1$ in $\mathcal{C}^1(B_{1/4})$. In particular, by step (4) of Lemma \ref{lem: mega lemma} we have $|\nabla v_1(0)|=\lim_n |\nabla v_{1,n}(0)|=1$, which completes the proof.
\end{proof}

\begin{remark}\label{rem: x_n to free boundary}
The sequence $x_n$ is bounded and thus, up to a subsequence, converges to some $\bar x \in \overline{B_2}$. Following \cite{ctv, nttv, stz}, it is possible to show that $\bar x$ has to be a free-boundary point, that is $\mf{u}(\bar x) = 0$: indeed, if this is not the case, then there exists $i$ such that $u_i(\bar x) \geq C > 0$. Using the local version of the main results in \cite{nttv} (which hold for general systems, and for which we refer to \cite{stz}), $\mf{u}_{\beta_n} \to \mf{u}$ in $\C^{0,\alpha}(\overline{B_2})$ for every $\alpha < 1$, and in particular $u_{i,\beta_n}(x) \geq C/2$ for every $x \in B_{2\delta}(\bar x)$ for some $\delta > 0$ and $\beta_n$ sufficiently large. Reasoning as in Lemma \ref{lem: non trivial limit}, this implies that $\mf{u}_{\beta_n} \to \mf{u}$ in $\C^{1,\alpha}(B_{\delta}(\bar x))$, a contradiction with the unboundedness of the gradient at $x_n$.
\end{remark}

Before concluding the section, we report some further properties of the blow-up sequences and of the asymptotic behaviour of the quantities previously introduced.

\begin{lemma}\label{lem: M_n bound from below}
There exists $C > 0$ such that $M_n \geq C$.
\end{lemma}
\begin{proof}
Let us assume by contradiction that there exists a subsequence $M_{n_k} \to 0$. By the previous results, the limiting function $\bv$ is made of entire harmonic functions which are bounded from below, thus constant thanks to the Liouville theorem: this contradicts the fact that $|\nabla v_1(0)| = 1$.
\end{proof}

\begin{lemma}\label{lem: at most 2}
Each limiting profile $\bv$ contains \emph{at most} two non trivial components.
\end{lemma}

The proof of the lemma is based upon the Alt-Caffarelli-Friedman monotonicity formula, as extended by Conti et al. in \cite{ctv}. We recall the results and some suitable generalizations whose proofs follow in a straightforward way and are thus omitted.

\begin{lemma}[Lemma 2.7 in \cite{ctv}]\label{lem: 2.7 in ctv}
Let $\bv \in \C(\R^N) \cap H^1_{\loc}(\R^N)$ be a vector of $k \geq 2$ non-trivial subharmonic functions such that $v_i v_j \equiv 0$ in $\R^N$ for every $i \neq j$, and there exists $x_0 \in \R^N$ such that $v_i(x_0) = 0$ for every $i$ . Then there exists $\nu(k,N) \geq 1$ such that the quantity
\[
    \Phi(r) := \prod_{i=1}^{k} \frac{1}{r^{2\nu(k,N)}} \int_{B_r(x_0)} \frac{|\nabla v_i|^2}{|x|^{N-2}} \de x
\]
is monotone non decreasing for $r>0$. If $k \geq 3$, then one can choose $\nu(k,N) > 1$.
\end{lemma}

\begin{corollary}[Hidden in Proposition 7.2 in \cite{ctv}]\label{corol 1 ACF}
Let $\bv$ as in the previous lemma. If there exists $C>0$ such that
\[
    |\bv(x)| \leq C (1+|x|),
\]
then $k \leq 2$.
\end{corollary}


\begin{lemma}[Lemma 2.7 in \cite{ctv}]
Let $\bv \in \C(\R^N) \cap H^1_{\loc}(\R^N)$ be a vector of $k \geq 2$ non-trivial positive functions, solutions to
\[
   \Delta v_i = v_i \tsum_{j \neq i}a_{ij} v_j^q \qquad \text{in $\R^N$}
\]
for $q \geq 1$. There exists $\nu(k,N) \geq 1$ such that for every $\gamma < \nu$ there exists $\bar r > 1$ such that the quantity
\[
    \Phi(r) := \prod_{i=1}^{k} \frac{1}{r^{2\gamma}} \int_{B_r} \left(|\nabla v_i|^2 + v_i^{2} \tsum_{j \neq i} a_{ij} v_j^q\right)|x|^{2-N}
\]
is monotone non decreasing for $r>\bar r$. If $k \geq 3$, then one can choose $\nu(k,N) > 1$.
\end{lemma}

\begin{corollary}[Hidden in Proposition 7.1 in \cite{ctv}]\label{corol acf 2}
Let $\bv$ as in the previous lemma. If there exists $C>0$ such that
\[
    |\bv(x)| \leq C (1+|x|),
\]
then $k \leq 2$. If moreover $\bv$ is non constant, then $k = 2$.
\end{corollary}

\begin{remark}
For a detailed proof of Corollary \ref{corol acf 2}, we refer to Corollary 1.14 in \cite{SoTe}. In an analogue way, the reader can derive Corollary \ref{corol 1 ACF} starting from Lemma \ref{lem: 2.7 in ctv}.
\end{remark}

We conclude this section by summing up what we proved so far in the following statement.

\begin{proposition}\label{prop: riassunto blow-up}
Let $\{\mf{u}_{\beta_n}\}$ satisfy the assumptions of Theorem \ref{thm: main general be} or \ref{thm: main general lv}, and assume that \eqref{absurd assumption} holds. Then the sequences $\{\mf{v}_n\}$ and $\{\bar{\mf{v}}_n\}$ defined by \eqref{def blow-up} have the properties (1)-(6) of Lemma \ref{lem: mega lemma}. There exists $C>0$ such that for every $i$
\[
v_i(x) \le C(1+|x|) \qquad \text{for every $x \in \R^N$},
\]
$\mf{v}$ is non-trivial and non-constant, and in particular $|\nabla v_1(0)|=1$. Moreover, $\mf{v}$ has at most $2$ non-trivial components, say $v_1$ and $v_2$, $M_n \ge C>0$, and
\begin{itemize}
\item if $(M_n)$ is bounded, then
\begin{equation}\label{eq intera}
\begin{cases}
-\Delta v_1= - M_\infty v_1 v_2^q & \text{in $\R^N$} \\
-\Delta v_2 = -M_\infty v_1^q v_2 & \text{in $\R^N$} \\
v_1,v_2 \ge 0 & \text{in $\R^N$},
\end{cases}
\end{equation}
where $M_n \to M_\infty$ as $n \to \infty$, and the convergence of $\mf{v}_n$ to $\mf{v}$ takes place in $\mathcal{C}^{1,\alpha}_{\loc}(\R^N)$ for every $\alpha<1$.
\item if $M_n \to +\infty$, then both $v_1$ and $v_2$ are subharmonic in $\R^N$, and
\begin{equation}\label{eq segregata}
\begin{cases}
-\Delta v_1= 0 & \text{in $\{v_1>0\}$} \\
-\Delta v_2 = 0 & \text{in $\{v_2>0\}$} \\
v_1 \cdot v_2 \equiv 0 & \text{in $\R^N$} \\
v_1,v_2 \ge 0 & \text{in $\R^N$}.
\end{cases}
\end{equation}
\end{itemize}
\end{proposition}
\section{Monotonicity formul{\ae}}\label{sec: monotonicity}

This section is devoted to some monotonicity formulae inspired by the Almgren frequency formula and the Alt-Caffarelli-Friedman monotonicity formula, which will be crucially employed in the proof of Theorem \ref{thm: main general be}. Some of the following results are already present in the literature, but not in the following generality, and hence we prefer to also prove them for the sake of completeness. In Subsection \ref{sub: ACF} we will use the assumption $N \ge 3$. As already explained, the case $N=2$ can be treated extending planar solutions as spacial ones, but we point out that it would be also possible to face directly the planar problem. This would require a slightly different Alt-Caffarelli-Friedman monotonicity formula inspired by Lemma 9.2 in \cite{ctvIndiana}, which we prefer to omit.


\subsection{Almgren monotonicity formul{\ae}}
Let us consider a smooth domain $\Omega \subset \R^N$, a compact set $K \subset\subset \Omega$ and a solution $\mf{u}=(u_1,\dots,u_k) \in H^1(\Omega)$ of a generic problem of type \eqref{be syst} satisfying the assumption \eqref{be assumpt}:
\[
\begin{cases}
-\Delta u_i= f_i(x,u_i) - \beta \sum_{j \neq i} a_{ij} u_i u_j^2  &\text{ in $\Omega$}\\
u_i \ge 0 &\text{ in $\Omega$,}
\end{cases}
\]
and there exists $m,d>0$ such that
\[
\|\mf{u}\|_{L^\infty(\Omega)} \le m \qquad \text{and} \qquad \max_i \sup_{0<s \le m} \left|\frac{f_{i}(x,s)}{s}\right| \le d.
\]
In what follows, all the constants that will appear \emph{depend} on the choice of $m$ and $d$, which are considered throughout these preliminary results as fixed, but \emph{are independent} on the choice of any other parameter (in particular, they are independent of $\beta > 0$). The reason behind this observation is that, in the next section, we aim at using monotonicity formulae for sequence of solutions to \eqref{be syst} which verify the assumptions in a uniform way.

For $x_0 \in K$ and $r>0$, we define
\begin{equation}\label{def N regular}
\begin{split}
  \bullet \quad & H(\mf{u},x_0,r):= \frac{1}{r^{N-1}} \int_{\partial B_r(x_0)} \sum_{i=1}^k u_i^2\\
  \bullet \quad & E(\mf{u},x_0,r):= \frac{1}{r^{N-2}} \int_{B_r(x_0)} \sum_{i=1}^k |\nabla u_i|^2+ 2\beta\sum_{1\le i<j\le k} a_{ij} u_i^2 u_j^2 - \sum_{i=1}^k f_i(x,u_i) u_i\\
  \bullet \quad & N(\mf{u},x_0,r):= \frac{E(\mf{u},x_0,r)}{H(\mf{u},x_0,r)} \qquad (\text{Almgren frequency function}).
\end{split}
\end{equation}

\begin{lemma}\label{lem: prelim_almgren}
Let $\mf{u}$ be a solution of \eqref{be syst} and \eqref{be assumpt}. For $x_0 \in K$ and $r>0$, we have
\begin{equation}\label{der di H}
\frac{\de}{\de r} H(\mf{u},x_0,r) = \frac{2}{r^{N-1}} \int_{\partial B_r(x_0)} \sum_{i=1}^k u_i \pa_\nu u_i = 2\frac{E(\mf{u},x_0,r)}{r}.
\end{equation}
Furthermore
\begin{align*}
\frac{\de}{\de r} E(\mf{u},x_0,r) & = \frac{2}{r^{N-2}} \int_{\pa B_r(x_0)} \sum_{i=1}^k (\pa_\nu u_i)^2 + \frac{(4-N)}{r^{N-1}} \beta \int_{B_r(x_0)} \sum_{1 \le i<j \le k} a_{ij} u_i^2 u_j^2  \\
& +\frac{1}{r^{N-1}} \int_{B_r(x_0)} \left[(N-2)\sum_{i=1}^k f_i(x,u_i)u_i + 2 \sum_{i=1}^k f_i(x,u_i) \nabla u_i \cdot (x-x_0) \right] \\
& + \frac{1}{r^{N-2}} \beta \int_{\pa B_r(x_0)} \sum_{1 \le i<j \le k}  a_{ij} u_i^2 u_j^2 - \frac{1}{r^{N-2}} \int_{\pa B_r(x_0)} \sum_{i=1}^k f_i(x,u_i)u_i.
\end{align*}
\end{lemma}
\begin{proof}
The equalities in \eqref{der di H} follow by direct computations. As far as the derivative of $E$ is concerned, we observe that
\begin{align*}
\frac{\de}{\de r} E(\mf{u},x_0,r) = &\frac{\de}{\de r}\left(\frac{1}{r^{N-2}} \int_{B_r(x_0)} \sum_i |\nabla u_i|^2+ 2\beta\sum_{i<j} a_{ij} u_i^2 u_j^2\right) + \frac{(N-2)}{r^{N-1}} \int_{B_r(x_0)} \sum_i f_i(x,u_i) u_i \\
& - \frac{1}{r^{N-2}} \int_{\pa B_r(x_0)} \sum_i f_i(x,u_i) u_i.
\end{align*}
To compute the first term on the right hand side, letting $u_{i,r}(x):= u_i(x_0+rx)$, we have
\begin{align*}
\frac{\de}{\de r}\left(\frac{1}{r^{N-2}} \int_{B_r(x_0)} \sum_i |\nabla u_i|^2+ 2\beta \sum_{i<j} a_{ij} u_i^2 u_j^2\right) = \frac{\de}{\de r}\left( \int_{B_1} \sum_i |\nabla u_{i,r}|^2+ 2 r^2 \beta \sum_{i<j} a_{ij} u_{i,r}^2 u_{j,r}^2\right) \\
= \int_{B_1} 2\sum_i \nabla u_{i,r} \cdot \nabla (\pa_r u_{i,r}) + 4r \beta\sum_{i<j} a_{ij} u_{i,r}^2 u_{j,r}^2 + 4r^2 \beta \sum_{i<j} a_{ij} u_{i,r} u_{j,r} \left( u_{j,r} \pa_r u_{i,r} + u_{i,r} \pa_r u_{j,r} \right)  \\
= 2\int_{\pa B_1}  \sum_i \pa_r u_{i,r} \pa_{\nu} u_{i,r} +\int_{B_1} 4r \beta\sum_{i<j} a_{ij} u_{i,r}^2 u_{j,r}^2 + 2r^2 \beta \sum_{i<j} a_{ij} u_{i,r} u_{j,r} \left( u_{j,r} \pa_r u_{i,r} + u_{i,r} \pa_r u_{j,r} \right) \\
+ \int_{B_1} 2r^2 \sum_i f_i(x_0+rx,u_{i,r}) \pa_r u_{i,r} \\
= \frac{2}{r^{N-2}} \int_{\pa B_r(x_0)} \sum_i  (\pa_\nu u_i)^2 + \frac{1}{r^{N-1}} \int_{B_r(x_0)} 4\beta\sum_{i<j} a_{ij} u_i^2 u_j^2  + \sum_{i<j} \nabla (u_i^2 u_j^2)\cdot (x-x_0) \\
+\frac{2}{r^{N-1}} \int_{B_r(x_0)}  \sum_i f_i(x,u_i) \nabla u_i \cdot (x-x_0).
\end{align*}
After a further integration by parts, the thesis follows.
\end{proof}

We recall the following formulation of the Poincar\'e inequality, which can be shown by a standard scaling argument.
\begin{lemma}[Poincar\'e inequality]\label{lem: poincare}
If $u \in H^1_{\loc}(\R^N)$, then the following inequality holds for any ball $B_r$:
\[
    \frac{1}{r^{N-2}} \int_{B_r} |\nabla u|^2 + \frac{1}{r^{N-1}} \int_{\partial B_r} u^2 \geq \frac{N-1}{r^{N}} \int_{B_r} u^2.
\]
\end{lemma}


\begin{lemma}\label{thm: almgren}
There exist two constants $\tilde r = \tilde r(m,d) >0$ and $\tilde C = \tilde C(m,d) >0$ such that
\[
N(\mf{u},x_0,r) + 1 \geq 0 \quad\text{and}\quad \frac{d}{dr}N(\mf{u},x_0,r) \ge - \tilde C(N(\mf{u},x_0,r) +1)
\]
for every $0<r \le \tilde r$, $x_0 \in K$.
\end{lemma}
\begin{proof}
Let us observe that, since by definition $H(\mf{u}, x_0, r) \geq 0$, the positivity of $N(\mf{u}, x_0,r) + 1$ is equivalent to that of $E(\mf{u}, x_0,r) +H(\mf{u}, x_0,r)$. By the sublinearity of $f_{i}$, we have
\begin{multline*}
    E(\mf{u}, x_0,r) +H(\mf{u}, x_0,r) \\= \frac{1}{r^{N-2}} \int_{B_r(x_0)} \sum_i |\nabla u_{i}|^2+ 2\beta\sum_{i<j} a_{ij} u_{i}^2 u_{j}^2 - \sum_i f_i(x,u_{i}) u_{i} + \frac{1}{r^{N-1}} \int_{\partial B_r(x_0)} \sum_i u_{i}^2 \\ \geq \frac{1}{r^{N-2}} \int_{B_r(x_0)} \sum_i |\nabla u_{i}|^2 -\frac{d r^2}{r^{N}} \int_{B_r(x_0)} \sum_i u_{i}^2 + \frac{1}{r^{N-1}} \int_{\partial B_r(x_0)} \sum_i u_{i}^2,
\end{multline*}
and thus we can conclude with an application of the Poincar\'e inequality in Lemma \ref{lem: poincare}, as long as $dr^2 \le d \tilde r^2<N-1$.

We now pass to the proof of the monotonicity, first dealing with the function $N(\mf{u}, x_0,r)$. We compute the derivative of $N$ using Lemma \ref{lem: prelim_almgren}. We have
\begin{multline*}
\frac{\de}{\de r} N(\mf{u},x_0,r)  = \frac{R(\mf{u},x_0,r)}{H(\mf{u},x_0,r)} \\
+\frac{2}{r^{2N-3} H^2(\mf{u},x_0,r)} \left[\left(\int_{\pa B_r(x_0)} \sum_i (\pa_\nu u_i)^2 \right) \left( \int_{\pa B_r(x_0)} \sum_i u_i^2 \right) - \left( \int_{\pa B_r(x_0)} \sum_i u_i \pa_{\nu} u_i \right)^2\right],
\end{multline*}
where
\begin{align*}
R(\mf{u},x_0,r) &:= \frac{(4-N) \beta }{r^{N-1}} \int_{B_r(x_0)}\sum_{i<j} a_{ij} u_i^2 u_j^2  \\
& \quad+\frac{1}{r^{N-1}} \int_{B_r(x_0)} \left[(N-2)\sum_i f_i(x,u_i)u_i + 2 \sum_i f_i(x,u_i) \nabla u_i \cdot (x-x_0) \right] \\
& \quad + \frac{\beta}{r^{N-2}} \int_{\pa B_r(x_0)} \sum_{i<j} a_{ij} u_i^2 u_j^2 - \frac{1}{r^{N-2}} \int_{\pa B_r(x_0)} \sum_i f_i(x,u_i)u_i \\ &\ge \frac{1}{r^{N-1}} \int_{B_r(x_0)} \left[(N-2)\sum_i f_i(x,u_i)u_i + 2 \sum_i f_i(x,u_i) \nabla u_i \cdot (x-x_0) \right] \\
& \quad - \frac{1}{r^{N-2}} \int_{\pa B_r(x_0)} \sum_i f_i(x,u_i)u_i =: R_1(\mf{u},x_0,r).
\end{align*}
Here we used the fact that $N \le 4$. Thus, by the Cauchy-Schwarz inequality
\begin{equation}\label{der N}
\frac{\de }{\de r} N(\mf{u},x_0,r) \ge \frac{R_1(\mf{u},x_0,r)}{H(\mf{u},x_0,r)}
\end{equation}
for every $r>0$. We now estimate the remainder $R_1$, using the assumptions on the reaction terms $f_i$. For every $x_0 \in K$ and $0<r\le 1$ such that $B_r(x_0) \subset \subset K$, it results that
\begin{multline*}
| R_1(\mf{u},x_0,r) |  \le \frac{1}{r^{N-1}} \int_{B_r(x_0)} \left[ (N-2) d \sum_i u_{i}^2  +2dr \sum_i u_{i} |\nabla u_{i}| \right] + \frac{d}{r^{N-2}} \int_{\pa B_r(x_0)} \sum_i u_{i}^2 \\
\le  \frac{d}{r^{N-2}} \int_{B_r(x_0)} \sum_i |\nabla u_{i}|^2 +  \frac{dr^2 + (N-2)dr}{r^N} \int_{B_r(x_0)} \sum_i u_{i}^2 + \frac{dr}{r^{N-1}} \int_{\pa B_r(x_0)} \sum_i u_{i}^2 \\
 \le C(d,N)\left[ \frac{1}{r^{N-2}} \int_{B_r(x_0)} \sum_i |\nabla u_{i}|^2 + \frac{1}{r^N} \int_{B_r(x_0)} \sum_i u_{i}^2 + \frac{1}{r^{N-1}} \int_{\pa B_r(x_0)} \sum_i u_{i}^2 \right]
\end{multline*}
The Poincar\'e inequality (Lemma \ref{lem: poincare}) can be used in order to estimate the last term, showing that there exist $\tilde C,\tilde r>0$ such that
\[
| R_1(\mf{u},x_0,r) |  \le \tilde C( E(\mf{u},x_0,r) + H(\mf{u},x_0,r) )
\]
for every $x_0 \in K$, $0<r \le \tilde r \le 1$. Coming back to \eqref{der N}, we obtain the desired conclusion.
\end{proof}

\begin{remark}\label{rem: dimension criticality}
In the whole proof of Theorem \ref{thm: main general be}, we use the assumption $N \le 4$ only in the previous lemma. As we have already observed in the introduction, such an assumption can be dropped in absence of reaction terms ($f_{i,\beta} \equiv 0$ for every $i$). In such case it is possible to replace the definition of $E(\mf{u},x_0,r)$ with
\[
\tilde E(\mf{u},x_0,r):= \frac{1}{r^{N-2}} \int_{B_r(x_0)} \sum_i |\nabla u_i|^2 + \beta \sum_{i<j} u_i^2 u_j^2,
\]
proving an Almgren monotonicity formula for the function $\tilde N:= \tilde E/H$ independently on the dimension $N$ (we refer to Proposition 5.2 in \cite{BeTeWaWe} for the details). The rest of the proof of Theorem \ref{thm: main general be} can be adapted with minor changes.\\
We also point out that for $p \neq 1$ the condition $N \le 4$ becomes $p \le 2(1+1/p)$.
\end{remark}

\begin{proposition}\label{prp: monotonia almgren}
There exist $\tilde r = \tilde r(m,d)>0$ and $\tilde C = \tilde C(m,d) >0$ such that the functions
\[
    r \mapsto (N(\mf{u},x_0,r)+1)e^{\tilde C r} \quad \text{and} \quad r \mapsto \left( \frac{1}{r^{N-1}} \int_{\pa B_r(x_0)} u_{i}^2 \right) e^{\tilde C r}
\]
are non-negative and monotone non-decreasing for $r \in (0,\tilde r]$, for every $x_0 \in K$ and $i=1,\dots,k$.
\end{proposition}

\begin{proof}
The first part is a straightforward consequence of Lemma \ref{thm: almgren}. For the second part, we use assumption \eqref{be assumpt} and the Poincar\'e inequality (Lemma \ref{lem: poincare}):
\begin{multline*}
    \frac{\de}{\de r} \left(\frac{1}{r^{N-1}} \int_{\partial B_r(x_0)} u_{i}^2 \right) \geq  \frac{2}{r^{N-1}} \int_{B_r(x_0)} |\nabla u_{i}|^2 - \frac{2dr}{r^{N}} \int_{B_r(x_0)} u_{i}^2 \\
    \geq \frac{2}{r}\frac{(N-1)-dr^2}{(N-1)r^{N-2}} \int_{B_r(x_0)} |\nabla u_{i}|^2 - \frac{2dr}{(N-1)r^{N-1}} \int_{\partial B_r(x_0)} u_{i}^2 \geq -  \frac{\tilde C'}{r^{N-1}} \int_{\partial B_r(x_0)} u_{i}^2,
\end{multline*}
where  the constant $\tilde C'$ depends only on $d$ and on $N$, and the last inequality holds as long as $r< \tilde r'(d,N)$ sufficiently small. Replacing, if necessary, $\tilde C$ and $\tilde r$ of Lemma \ref{thm: almgren} with $\max\{\tilde C,\tilde C'\}$ and $\min\{\tilde r,\tilde r'\}$, the thesis follows by a further integration.
\end{proof}

We complete the first part of this subsection with two useful doubling properties.

\begin{lemma}\label{lem: doubling}
Let $\tilde C$ and $\tilde r$ be defined in the previous lemma.
\begin{itemize}
\item[($i$)] If there exist $0<\underline r<\bar r<\tilde r$ and $d>0$ such that $N(\mf{u},0,r) \le d$ for every $\underline r\le r \le \bar r$, then
\[
r \mapsto \frac{H(\mf{u},x_0,r)}{r^{2d}} \quad \text{is monotone non-increasing for $\underline r\le r \le \bar r$}.
\]
\item[($ii$)] If there exist $0<\underline r<\bar r<\tilde r$ and $\gamma>0$ such that $N(\mf{u},0,r) \ge \gamma$ for every $\underline r\le r \le \bar r$, then
\[
r \mapsto \frac{H(\mf{u},x_0,r)}{r^{2\gamma}} \quad \text{is monotone non-decreasing for $\underline r\le r \le \bar r$}.
\]
\end{itemize}
\end{lemma}
\begin{proof}
($i$) By \eqref{der di H} we observe that
\[
\frac{\de}{\de r}\log H(\mf{u},x_0,r) = \frac{2}{r}N(\mf{u},x_0,r) \le \frac{2d}{r}
\]
for every $\underline{r}\le r\le \bar r$. By integrating, the thesis follows. The proof of ($ii$) is analogue.
\end{proof}

\paragraph{\textbf{Almgren monotonicity formulae for segregated configurations}} In \cite{tt} the authors introduced the sets $\mathcal{G}(\Omega)$ and $\mathcal{G}_{\loc}(\Omega)$, classes of segregated vector valued functions sharing several properties with solutions of competitive systems, including a version of the Almgren monotonicity formula. We report Definition 1.2 in \cite{tt}, which is of interest in the present setting.

\begin{definition}\label{def: G}
For an open set $\Omega \subset \R^N$, we define the class $\mathcal{G}(\Omega)$ of non-trivial functions $\mf{0} \neq \mf{v}=(v_1,\dots,v_k)$ whose components are non-negative and locally Lipschitz continuous in $\Omega$, and such that the following properties holds:
\begin{itemize}
\item $v_i v_j \equiv 0$ in $\Omega$ for every $i \neq j$;
\item for every $i$
\[
-\Delta v_i= f_i(x,v_i)-\mu_i \qquad \text{in $\Omega$ in distributional sense},
\]
where $\mu_i$ is a non-negative Radon measure supported on the set $\pa\{v_i>0\}$, and $f_i: \Omega \times \R^+ \to \R$ are $\mathcal{C}^1$ functions such that $|f_i(x,s)| \leq d|s|$, uniformly in $x$;
\item defining for $x_0 \in \Omega$ and $r>0$ such that $B_r(x_0) \subset \Omega$ the function
\begin{equation}\label{def E segregated}
E(\mf{v},x_0,r):= \frac{1}{r^{N-2}} \int_{B_r(x_0)} \sum_{i=1}^k |\nabla v_i|^2- \sum_{i=1}^k f_i(x,v_i) v_i
\end{equation}
we assume that $E$ is absolutely continuous as function of $r$ and
\begin{align*}
\frac{\de}{\de r} E(\mf{v},x_0,r) &= \frac{1}{r^{N-2}} \int_{B_r(x_0)} \sum_{i=1}^k (\pa_\nu v_i)^2 - \frac{1}{r^{N-2}} \int_{\pa B_r(x_0)} \sum_{i=1}^k f_i(x,u_i)u_i\\
& \quad + \frac{1}{r^{N-1}} \int_{B_r(x_0)} \left[(N-2)\sum_{i=1}^k f_i(x,u_i)u_i + 2 \sum_{i=1}^k f_i(x,u_i) \nabla u_i \cdot (x-x_0) \right]  .
\end{align*}
\end{itemize}
For points $x_0 \in \{\mf{v}=\mf{0}\}$, we define the multiplicity of $x_0$ as
\[
    \sharp \left\{i=1,\dots,k: \text{$\meas\{B_r(x_0) \cap \{v_i>0\}\} >0$ for every $r>0$}\right\}.
\]
We write that $\mf{v} \in \mathcal{G}_{\loc}(\Omega)$ if $\mf{v} \in \mathcal{G}(K)$ for every compact set $K \subset \subset \Omega$.
\end{definition}

\begin{remark}
The definition of $E$ in \eqref{def N regular} and \eqref{def E segregated} are different, but we do not think that this can be source of misunderstanding, because the correct choice of $E$ is clearly determined by the vector valued function $\mf{v}$ which is considered. In the same spirit, we define the Almgren frequency function for elements of $\mathcal{G}(\Omega)$ as
\[
N(\mf{v},x_0,r):= \frac{E(\mf{v},x_0,r)}{H(\mf{v},x_0,r)},
\]
with $H$ defined as in \eqref{def N regular}.
\end{remark}

We recall some known facts. The following are a monotonicity formula for functions of $\mathcal{G}(\Omega)$, and a lower estimate of $N(\mf{v},x_0,0^+)$ for points $x_0$ on the free boundary $\{\mf{v}=\mf{0}\}$, for which we refer to Theorem 2.2 and Corollary 2.7 in \cite{tt}.

\begin{theorem}\label{prop: monot segregated}
Let $\mf{v} \in \mathcal{G}(\Omega)$ and let $K \subset \subset \Omega$. There exists $\tilde r',\tilde C'$ depending only on $d$ and on the dimension $N$, such that for every $x_0 \in K$ and $r \in (0,\tilde r']$ it results that $H(\mf{v},x_0,r) \neq 0$, the function $N(\mf{v},x_0,r)$ is absolutely continuous in $r$ and
\[
r \mapsto (N(\mf{v},x_0,r)+1)e^{\tilde C' r} \quad \text{is monotone non-decreasing}.
\]
Moreover, for every point of the free boundary $x_0 \in \{\mf{v}=\mf{0}\}$ it results that $N(\mf{v},x_0,0^+) \ge 1$.
\end{theorem}

\begin{remark}\label{rem: on unique continuation}
The fact that $H(\mf{v},x_0,r) \neq 0$ for $r \in (0,\tilde r']$ and $x_0 \in K$ is a unique continuation property for elements of $\mathcal{G}(\Omega)$: indeed, if in an open subset of $\Omega$ we have $\mf{v} \equiv \mf{0}$, then $H(\mf{v},x_0,r) \equiv 0$ for some $x_0 \in \Omega$ and $r \in (r_1,r_2)$, in contradiction with the previous result.
\end{remark}

The almost monotonicity formula for $N$ becomes a full monotonicity formula if $f_{i} \equiv 0$ for every $i$ (see Remark 2.4 in \cite{tt}). Moreover, thanks to a classification result due to \cite{nttv} (see Step 6 in Proposition 3.9), the following holds.

\begin{proposition}\label{prop: monot segregated senza f}
Let $\mf{v} \in \mathcal{G}(\Omega)$ with $f_i \equiv 0$ for every $i$. Then $r \mapsto N(\mf{v},x_0,r)$ is non-decreasing. Moreover, it holds $N(\mf{v},x_0,r) \equiv \sigma>0$ for $r \in (0,\bar r]$ if and only if $\mf{v}$ is a non-trivial homogeneous function of degree $\sigma$.
\end{proposition}

The relation between solutions of strongly competing systems and functions in $\mathcal{G}(\Omega)$ is clarified by the following statement, for which we refer to Theorem 8.1 in \cite{tt} in case $f_{i,\beta}(u_i):=\omega_i u_i^3- \lambda_{i,\beta} u_i$, and to \cite{stz} in a completely general setting.

\begin{proposition}\label{prop: limiting profile}
Let us a consider a sequence $\beta \to +\infty$, and let $\{\mf{u}_{\beta}\}$ be a corresponding sequence of solution to \eqref{be syst} in $\Omega$ satisfying \eqref{be assumpt} independently on $\beta$. Assume that
\[
f_{i,\beta} \to f_{i} \qquad \text{in $\mathcal{C}_{\loc}(\Omega \times [0,m])$}
\]
for some $f_i \in \mathcal{C}^1(\Omega \times [0,m])$, and that there exists $\mf{u}$ such that
\[
\mf{u}_\beta \to \mf{u} \qquad \text{in $\mathcal{C}(\overline{\Omega}) \cap H^1(\Omega)$}.
\]
Then $\mf{u} \in \mathcal{G}(\Omega)$, and $N(\mf{u}_\beta, x, r) \to N(\mf{u}, x, r)$ for every $x \in \Omega$ and $r > 0$ such that $B_r(x_0) \subset \subset \Omega$.
\end{proposition}

If $\mf{u}$ is as in the previous theorem, we write that \emph{$\mf{u}$ is a limiting profile of system \eqref{be syst}} (as $\beta \to +\infty$). A result which will be crucially employed in the rest of the section establishes the non occurrence of self-segregation for limiting profiles of strongly competing systems. This has been proved in Section 10 of \cite{dwz1}.

\begin{theorem}\label{prop: on multipl}
Let $\mf{v} \in \mathcal{G}(\Omega)$ be a limiting profile of system \eqref{be syst}, and let $x_0 \in \{\mf{v}=\mf{0}\}$. Then $x_0$ has multiplicity greater than or equal to $2$.
\end{theorem}

\subsection{A perturbed Alt-Caffarelli-Friedman monotonicity formula}\label{sub: ACF}

We introduce in a general setting an Alt-Caffarelli-Friedman monotonicity formula which has been proved for the first time in a specific situation in \cite{Wa}, Theorem 4.3; accordingly to the current section, here we consider the case $N \ge 3$.

We consider two components, say $u_1$ and $u_2$, of a solution $\mf{u}$ of system \eqref{be syst}:
\[
\begin{cases}
-\Delta u_i= f_i(x,u_i) - \beta \sum_{j \neq i} a_{ij} u_i u_j^2  &\text{ in $\Omega$}\\
u_i > 0 &\text{ in $\Omega$}.
\end{cases}
\]
The ingredients of our result are the following:
\begin{equation}\label{def gamma e lambda}
\begin{split}
\bullet \quad & J_{1}(r) := \int_{B_r} \left(\left|\nabla u_{1}\right|^2 + \beta a_{12} u_{1}^2 u_{2}^2-u_{1} f_{1}(x,u_{1}) \right)|x|^{2-N}  \\
\bullet \quad &  J_{2}(r) := \int_{B_r} \left(\left|\nabla u_{2}\right|^2 + \beta a_{12} u_{1}^2 u_{2}^2-u_{2} f_{2}(x,u_{2}) \right)|x|^{2-N} \\
\bullet \quad &  \gamma(t) := \sqrt{\left(\frac{N-2}{2}\right)^2+t}-\frac{N-2}{2} \\
\bullet \quad &  \Lambda_{1}(r):= \frac{r^2\int_{\pa B_r} \left|\nabla_\theta u_{1}\right|^2 + \beta a_{12} u_{1}^2 u_{2}^2-u_{1} f_{1}(x,u_{1})  }{\int_{\pa B_r} u_{1}^2} \\
\bullet \quad &  \Lambda_{2}(r):= \frac{r^2\int_{\pa B_r} \left|\nabla_\theta u_{2}\right|^2 + \beta a_{12} u_{1}^2 u_{2}^2-u_{2} f_{2}(x,u_{2}) }{\int_{\pa B_r} u_{2}^2}.
\end{split}
\end{equation}

\begin{theorem}\label{thm: ACF}
Let $\mf{u}$ be a solution of \eqref{be syst} and let $R > 1$, $\lambda,\mu,\eps>0$ be such that
\begin{itemize}
\item[($h0$)] $\eps R^2 \leq \left(\frac{N-2}{2}\right)^2$;
\item[($h1$)] $J_i(r), \Lambda_i(r) >0$ for every $r \in (1,R)$, for $i=1,2$;
\item[($h2$)] it holds
\[
\frac{1}{\lambda} \le \frac{\int_{\pa B_r} u_1^2}{\int_{\pa B_r} u_2^2} \le \lambda \quad \text{and} \quad \frac{1}{r^{N-1}} \int_{\pa B_r} u_i^2 \ge \mu
\]
for every $r \in (1,R)$, $i=1,2$;
\item[($h3$)] $|f_i(x,u_i)| \le \eps u_i$ in $\Omega$ for $i=1,2$.
\end{itemize}
There exists a positive constant $C>0$, depending only on $\lambda,\mu$ and on the dimension $N$, such that
\[
    r \mapsto \frac{J_1(r) J_2(r)}{r^4} \exp\{-C (\beta r^2)^{-1/4} + C \eps r^2\} \qquad \text{is monotone non-decreasing for $r \in (1,R)$}.
\]
\end{theorem}

\begin{remark}
For future convenience, we point out that the constant $C$ of the thesis is independent by the ends of the interval $(1,R)$.
\end{remark}

The proof rests upon the following lemma, which can be seen as a Poincar\'e lemma on the sphere $\S^{N-1}$, $N \geq 3$ for two competing densities. This result is actually a generalization of Lemma 4.2 in \cite{Wa}.

For any $\lambda>0$, let
\[
H_\lambda:= \left\{(u,v) \in \left(H^1(\S^{N-1})\right)^2: \int_{\S^{N-1}} u^2 =1 \quad \text{and} \quad \int_{\S^{N-1}} v^2 = \lambda \right\}.
\]

\begin{lemma}\label{lem: lemma 4.2 di wang}
Let us fix any $\bar \lambda > 1$. There exists $C = C(N, \bar \lambda)$ such that if
\[
    \frac{1}{\bar \lambda} < \lambda < \bar \lambda,\quad k > 0 \quad \text{and} \quad 0 \leq \eps \leq \left(\frac{N-2}{2}\right)^2
\]
then
\begin{multline}\label{minimiz problem 1}
\min_{(u,v) \in H_{\lambda}} \gamma\left( \int_{\S^{N-1}} |\nabla_\theta u|^2 + k u^2 v^2 -\eps u^2\right) + \gamma\left( \frac{\int_{\S^{N-1}} |\nabla_\theta v|^2 + k u^2 v^2  -\eps v^2}{\int_{\S^{N-1}} v^2}\right) \\
\ge 2- C\left(\eps+ k^{-1/4}\right).
\end{multline}
\end{lemma}
\begin{proof}
It is straightforward to check that the estimate of the lemma is equivalent to related one for the functional
\[
    J(u,v):= \gamma\left( \int_{\S^{N-1}} \left(|\nabla_\theta u|^2 + k \lambda u^2 v^2\right) -\eps\right) + \gamma\left(\int_{\S^{N-1}} \left(|\nabla_\theta v|^2 + k u^2 v^2\right) - \eps \right)
\]
considered on the $H^1$-weakly closed set $H_1$. First of all, we point out that such a minimization problem is well posed: indeed the domain of the function $t \mapsto \gamma(t)$ is given by the half line $t\geq -\left(\frac{N-2}{2}\right)^2$, and the restriction on $\eps$ is sufficient to ensure the meaningfulness of \eqref{minimiz problem 1}. Moreover, for $k \geq 0$ the functional $J$ is coercive and lower-semicontinuous in the weak topology of $H^1(\S^{N-1})$: thus for any triplet $(\lambda, \eps, k)$ that satisfies the assumptions, the minimization problem admits a solution. To conclude the lemma we only need to check the asymptotic expansion of the right hand side of \eqref{minimiz problem 1} for $\eps$ small and $k$ large.

Let us a consider a sequence of triplets $(\lambda_n, \eps_n, k_n)$ satisfying the assumptions and such that $\eps_n \to 0$ and $k_n \to +\infty$, and for any such triplet let us consider a minimizer $(u_n, v_n)$ of the functional $J$. As $J(u,v)=J(|u|,|v|)$, it is not restrictive to assume that $u_n, v_n \ge 0$ in $\mathbb{S}^{N-1}$. Moreover, thanks to Lemma 4.1 in \cite{Wa} it is possible to check that the functional $J$ is decreasing with respect to antipodal Steiner symmetrization rearrangements of the functions $(u_n, v_n)$, and thus we can also assume that the minimizer depends only on the angular coordinate on the sphere $\alpha \in [0,\pi]$, and that $u_n$ is decreasing while $v_n$ is increasing in $\alpha$. Let
\[
    x_n = \int_{\S^{N-1}} \left(|\nabla_\theta u_n|^2 + k_n \lambda_n u_n^2 v_n^2\right) -\eps_n \quad \text{and} \quad y_n = \int_{\S^{N-1}} \left(|\nabla_\theta v_n|^2 + k_n  u_n^2 v_n^2\right) -\eps_n.
\]
By the Lagrange multipliers rule, there exist $\mu_{1,n},\mu_{2,n} \in \R$ such that
\begin{equation}\label{eq per minimi}
    \begin{cases}
        -\Delta_\theta u_n = -k_n \left(\lambda_n +\frac{\gamma'(y_n)}{\gamma'(x_n)}\right) u_n v_n^2 +  \frac{\mu_{1,n}}{\gamma'(x_n)} u_n  \\
        -\Delta_\theta v_n = -k_n \left(1 +\frac{\gamma'(x_n)}{\gamma'(y_n)} \lambda_n\right) u_n v_n^2 +  \frac{\mu_{2,n}}{\gamma'(y_n)} v_n
    \end{cases} \qquad \text{in $\S^{N-1}$}
\end{equation}
where $\Delta_\theta$ is the Laplace-Beltrami operator on the sphere $\S^{N-1}$. Since $(u_n, v_n)$ depends only on one angular coordinate, \eqref{eq per minimi} simplifies as
\[
    \begin{cases}
        -\frac{\de^2 u_n}{\de \alpha^2}  -(N-2) \cot \alpha \frac{\de u_n}{\de \alpha}  = - k_n \left( \lambda_n + \frac{\gamma'(y_n)}{\gamma'(x_n)} \right)u_n v_n^2 + \frac{\mu_{1,n}}{\gamma'(x_n)} u_n \\
        -\frac{\de^2 v_n}{\de \alpha^2}  -(N-2) \cot \alpha \frac{\de v_n}{\de \alpha}  = - k_n \left( 1 + \frac{\gamma'(x_n)}{\gamma'(y_n)} \lambda_n\right)v_n u_n^2 + \frac{\mu_{2,n}}{\gamma'(y_n)} v_n
    \end{cases} \qquad \text{in $[0,\pi]$}.
\]
Note that, with respect to Lemma 4.2 in \cite{Wa}, the presence of $\eps_n$ is irrelevant for the characterization of $(u_n,v_n)$. As a consequence, it is possible to repeat step by step the proof of the quoted result, and to conclude that:
\begin{itemize}
\item the sequence $(J_n(u_n,v_n))$ is bounded. Thus $\{(u_n,v_n)\}$ is bounded in $H^1(\S^{N-1})$, and $(\gamma'(x_n))$, $(\gamma'(y_n))$ are bounded from above and from below by positive constants. Moreover, there exists $C>0$ independent of $n$ such that
\[
\int_{\S^{N-1}} k_n u_n^2 v_n^2 \le C.
\]
\item the sequences of the Lagrange multipliers $(\mu_{1,n})$ and $(\mu_{2,n})$ are bounded, and by a Brezis-Kato argument together with equation \eqref{eq per minimi} and the $H^1$-boundedness, this implies that $\{(u_n,v_n)\}$ is bounded in $L^\infty(\S^{N-1})$.
\item There exists $\bar \alpha_n \in (0,\pi)$ such that $\{u_n > v_n\} = \{\theta \in \S^{N-1}: \alpha < \bar \alpha_n\}$, $\{u_n = v_n\} = \{\theta \in \S^{N-1}: \alpha = \bar \alpha_n\}$ and $\{u_n < v_n\} = \{\theta \in \S^{N-1}: \alpha > \bar \alpha_n\}$; up to multiplicative constants (depending on $n$) $u_n \to (\cos\alpha)^+$, $v_n \to (\cos\alpha)^-$ in $H^1(\S^{N-1})\cap \C(\S^{N-1})$, and in particular $\bar \alpha_n \to \pi/2$.
    \item There exists a constant $C>0$, independent of $\eps_n$ and $k_n$, such that the Lipschitz norm of $(u_n,v_n)$ is smaller than $C$ (implied by a small modification of Lemma 2.4 in \cite{blwz_phase}).
    \item The following pointwise estimate holds uniformly in $n$:
    \begin{equation}\label{optimal decay}
    u_n v_n \le C k_n^{-1/2} \qquad \text{in $\S^{N-1}$}.
    \end{equation}
\end{itemize}
The decay estimate \eqref{optimal decay} implies that for every $\theta \in \S^{N-1}$, either $u_n(\theta) \le C k_n^{-1/4}$ or $v_n(\theta) \le C k_n^{-1/4}$. Let us introduce the functions $f_n=(u_n-v_n)^+$ and $g_n=(u_n-v_n)^-$. As in \cite{Wa}, by \eqref{optimal decay}
\[
\int_{\S^{N-1}} |f_n-u_n|^2  \le C k_n^{-1/2},
\]
and
\[
\begin{split}
\int_{\S^{N-1}} |\nabla_\theta f_n|^2 & \le  \int_{\{u_n>v_n\}} \left(|\nabla_\theta u_n|^2 + \lambda_n k_n u_n^2 v_n^2\right) + C k_n^{-1/4} \\
& \le \int_{\{u_n>v_n\}} \left(|\nabla_\theta u_n|^2 + \lambda_n k_n u_n^2 v_n^2\right) \pm \eps_n + C k_n^{-1/4}.
\end{split}
\]
Therefore
\begin{align*}
\frac{\int_{\S^{N-1}} |\nabla_\theta f_n|^2}{\int_{\S^{N-1}} f_n^2} & \le \frac{ \int_{\S^{N-1}} \left(|\nabla_\theta u_n|^2 + \lambda_n k_n u_n^2 v_n^2\right) \pm \eps_n + C k_n^{-1/4}}{\int_{\S^{N-1}} u_n^2 -C k_n^{-1/2}} \\
& = \frac{x_n +\eps_n + C k_n^{-1/4}}{1-C k_n^{-1/2}} \le x_n +\eps_n+ C k_n^{-1/4}.
\end{align*}
A similar estimate holds for $g_n$. By the monotonicity and the concavity of $\gamma$, we finally infer
\begin{multline*}
2 \le \gamma\left(\frac{\int_{\S^{N-1}} |\nabla_\theta f_n|^2}{\int_{\S^{N-1}} f_n^2}\right) + \gamma\left(\frac{\int_{\S^{N-1}} |\nabla_\theta g_n|^2}{\int_{\S^{N-1}} g_n^2}\right) \\ \le \gamma\left(x_n + \eps_n+ C k_n^{-1/4}\right) + \gamma\left(y_n+\eps_n + C k_n^{-1/4}\right) \le \gamma(x_n) + \gamma(y_n) +C \eps_n+ C k_n^{-1/4},
\end{multline*}
where the first inequality is a consequence of the fact that $f_n$ and $g_n$ have disjoint support (this is nothing but the well-known optimal partition problem of the sphere which serves as keystone in the proof of the original Alt-Caffarelli-Friedman monotonicity formula, see for instance \cite{acf, csbook}).
\end{proof}

\begin{remark}
In the proof of Lemma 4.2 in \cite{Wa}, the author makes use of the exponential decay estimate for solution of \eqref{be syst}. The reader can easily check that such a decay can be replaced by the polynomial one proved in Lemma \ref{lem: polynomial decay}. This allows to generalize the previous lemma, showing that for any $p \ge 1$, under the same assumptions on $\lambda$, $\eps$ and $k$, there exists $C>0$ such that
\begin{multline*}
\min_{(u,v) \in H_{\lambda}} \gamma\left( \int_{\S^{N-1}} |\nabla_\theta u|^2 + k u^{p+1}v^{p+1} -\eps u^2\right) + \gamma\left( \frac{\int_{\S^{N-1}} |\nabla_\theta v|^2 + k u^{p+1} v^{p+1}  -\eps v^2}{\int_{\S^{N-1}} v^2}\right) \\
\ge 2- C\left(\eps+ k^{-1/(2p+2)}\right).
\end{multline*}
\end{remark}

Now a technical result.

\begin{lemma}\label{lem: positivita}
Let $i=1,2$, and let $r>0$ be such that $J_{i}(r) >0$ and $\Lambda_{i}(r)>0$. Then
\[
J_{i}(r) \le \frac{r}{2\gamma(\Lambda_{i}(r))} \int_{\pa B_r} \left(\left|\nabla u_{i}\right|^2 + \beta a_{12} u_{1}^2 u_{2}^2-u_{i} f_{i}(x,u_{i}) \right)|x|^{2-N}.
\]
\end{lemma}
\begin{proof}
We consider $i=1$. By testing the equation for $u_{1}$ against $u_{1} |x|^{2-N}$ in $B_r$, we obtain
\begin{align*}
J_{1}(r) & = -\frac12 \int_{B_r} \nabla(u_{1}^2)\cdot \nabla (|x|^{2-N}) + \frac{1}{r^{N-2}}\int_{\pa B_r} u_{1} \pa_{\nu} u_{1} \\
&= \frac{1}{2}\int_{B_r} u_{1}^2\Delta(|x|^{2-N}) + \frac{1}{r^{N-2}}\int_{\pa B_r} u_{1} \pa_{\nu} u_{1} +\frac{N-2}{r^{N-1}} \int_{\pa B_r} u_{1}^2 \\
& \le \frac{1}{r^{N-2}}\int_{\pa B_r} u_{1} \pa_{\nu} u_{1} +\frac{N-2}{r^{N-1}} \int_{\pa B_r} u_{1}^2,
\end{align*}
where the last inequality follows from the fact that $\Delta(|x|^{2-N}) =-C \delta$ for some dimensional constant $C>0$, where $\delta$ is the Dirac delta centred in $0$. Now, by Cauchy-Schwarz and Young inequalities, we have
\begin{align*}
\int_{\pa B_r} u_{1} \pa_{\nu} u_{1} &\le \left(\int_{\pa B_r}  u_{1}^2\right)^{1/2} \left(\int_{\pa B_r}  (\pa_{\nu} u_{1})^2\right)^{1/2}  \\
& \le \frac{\gamma(\Lambda_{1}(r))}{2r} \int_{\pa B_r} u_{1}^2 + \frac{r}{2\gamma(\Lambda_{1}(r))} \int_{\pa B_r} (\pa_\nu u_{1})^2,
\end{align*}
where we used the fact that $\Lambda_{1}(r)>0$. Plugging this estimate in the previous chain of inequalities and using the definition of $\gamma$, we deduce that
\begin{align*}
J_{1}(r) & \le \frac{1}{2r^{N-1}\gamma(\Lambda_{1}(r)) } \left[ \left( \gamma(\Lambda_{1}(r))^2+(N-2) \gamma(\Lambda_{1}(r)) \right) \int_{\pa B_r} u_{1}^2 + r^2 \int_{\pa B_r} (\pa_\nu u_{1})^2\right] \\
& \le \frac{r^2}{2r^{N-1}\gamma(\Lambda_{1}(r)) } \left[ \int_{\pa B_r} \left(|\nabla_\theta u_{1}|^2+ \beta a_{12} u_{1}^2 u_{2}^2 - u_{1} f_{1}(x,u_{1})\right)+ \int_{\pa B_r} (\pa_\nu u_{1})^2\right],
\end{align*}
which is the desired result.
\end{proof}

\begin{proof}[Proof of Theorem \ref{thm: ACF}]
Since for $r \in (1,R)$ both $J_i(r)$ and $\Lambda_i(r)$ are strictly positive, we can compute the logarithmic derivative of $J_1(r) J_2(r)/r^4$ and apply Lemma \ref{lem: positivita}, deducing that
\begin{align*}
\frac{\de}{\de r} \log\left( \frac{J_1(r) J_2(r)}{r^4}\right) &= -\frac{4}{r}+ \frac{\int_{\pa B_r} \left(\left|\nabla u_{1}\right|^2 + \beta a_{12} u_{1}^2 u_{2}^2-u_{1} f_{1}(x,u_{1}) \right)|x|^{2-N}}{\int_{B_r} \left(\left|\nabla u_{1}\right|^2 + \beta a_{12} u_{1}^2 u_{2}^2-u_{1} f_{1}(x,u_{1}) \right)|x|^{2-N}} \\
& \hphantom{ \ge -\frac{4}{r}} + \frac{\int_{\pa B_r} \left(\left|\nabla u_{2}\right|^2 + \beta a_{12} u_{1}^2 u_{2}^2-u_{2} f_{2}(x,u_{2}) \right)|x|^{2-N}}{\int_{B_r} \left(\left|\nabla u_{2}\right|^2 + \beta a_{12} u_{1}^2 u_{2}^2-u_{2} f_{2}(x,u_{2}) \right)|x|^{2-N}} \\
& \ge \frac{2}{r}\left(\gamma(\Lambda_{1}(r))+\gamma(\Lambda_{2}(r))-2\right) \\
& \ge -\frac{4}{r} + \frac{2}{r} \gamma\left(\frac{r^2 \int_{\pa B_r} |\nabla_\theta u_{1}|^2 + \beta a_{12} u_{1}^2 u_{2}^2-  \eps u_{1}^2}{\int_{\pa B_r} u_{1}^2 }\right) \\
& \hphantom{\ge -\frac{4}{r}} + \frac{2}{r} \gamma\left(\frac{r^2 \int_{\pa B_r} |\nabla_\theta u_{2}|^2 + \beta a_{12} u_{1}^2 u_{2}^2- \eps u_{2}^2}{\int_{\pa B_r} u_{2}^2 }\right),
\end{align*}
where we used assumption ($h3$) and the monotonicity of $\gamma$. The idea is now to apply Lemma \ref{lem: lemma 4.2 di wang} on the right hand side, and in order to do this we introduce
\[
u_{i,r}(x):=\frac{u_{i}(rx)}{\left( \frac{1}{r^{N-1}}\int_{\pa B_r} u_{1}^2 \right)^{1/2}} \qquad i=1,2;
\]
we emphasize that both $u_1$ and $u_2$ are normalized with respect to the average of $u_1$. By direct computations
\begin{align*}
\frac{r^2 \int_{\pa B_r} |\nabla_\theta u_{1}|^2 + \beta a_{12} u_{1}^2 u_{2}^2- \eps u_{1}^2}{\int_{\pa B_r} u_{1}^2 } &= \int_{\pa B_1} |\nabla_\theta u_{1,r}|^2 + r^2 \left(\frac{1}{r^{N-1}}\int_{\pa B_r} u_{1}^2 \right) \beta a_{12} u_{1,r}^2 u_{2,r}^2 - \eps r^2 u_{1,r}^2 \\
\frac{r^2 \int_{\pa B_r} |\nabla_\theta u_{2}|^2 + \beta a_{12} u_{1}^2 u_{2}^2- \eps u_{2}^2}{\int_{\pa B_r} u_{2}^2 } &= \frac{\int_{\pa B_1} |\nabla_\theta u_{2,r}|^2 + r^2 \left(\frac{1}{r^{N-1}}\int_{\pa B_r} u_{1}^2 \right) \beta a_{12} u_{1,r}^2 u_{2,r}^2 - \eps r^2 u_{2,r}^2 }{\int_{\pa B_1}  u_{2,r}^2  }.
\end{align*}
Thanks to assumptions ($h0$) and ($h2$)
\begin{align*}
\eps r^2 \le \eps R^2 \le \left(\frac{N-2}{2}\right)^2\\
\int_{\pa B_1} u_{1,r}^2 =1 \quad \text{and} \quad \frac{1}{\lambda} \le \int_{\pa B_1} u_{2,r}^2 \le \lambda\\
\frac{1}{r^{N-1}}\int_{\pa B_r} u_{1}^2  \ge \mu
\end{align*}
for every $r \in (1,R)$ and $i=1,2$. Therefore we are in position to apply Lemma \ref{lem: lemma 4.2 di wang}, obtaining
\begin{align*}
\frac{\de}{\de r} \log\left( \frac{J_1(r) J_2(r)}{r^4}\right) &\ge -\frac{4}{r} + \frac{2}{r} \gamma\left(\int_{\pa B_1} |\nabla_\theta u_{1,r}|^2 + r^2 \mu \beta a_{12} u_{1,r}^2 u_{2,r}^2 - \eps r^2 u_{1,r}^2\right) \\
& \qquad \quad + \frac{2}{r}\gamma\left(\frac{\int_{\pa B_1} |\nabla_\theta u_{2,r}|^2 + r^2 \mu \beta a_{12} u_{1,r}^2 u_{2,r}^2 - \eps r^2 u_{2,r}^2 }{\int_{\pa B_1}  u_{2,r}^2  }\right) \\
& \ge -C \beta^{-1/4} r^{-3/2} - C \eps r.
\end{align*}
By integrating, the thesis follows.
\end{proof}

\section{Interior Lipschitz bound in the variational setting for $N \ge 3$}\label{sec: variational}

In this section we complete the proof of Theorem \ref{thm: main general be}, obtaining a contradiction between the conclusions of the blow-up analysis, Proposition \ref{prop: riassunto blow-up}, and the assumption $L_n \to +\infty$. We recall that we are considering a sequence $\{\mf{u}_{\beta_n}\}$ of solutions to \eqref{be syst} satisfying the assumptions of Theorem \ref{thm: main general be}, with $\beta_n \to +\infty$ :
\[
\begin{cases}
-\Delta u_{i,\beta_n}= f_{i,\beta_n}(x,u_{i,\beta_n}) - \beta_n \sum_{j \neq i} a_{ij} u_{i,\beta_n} u_{j,\beta_n}^2  &\text{ in $\Omega$}\\
u_{i,\beta_n} > 0 &\text{ in $\Omega$,}
\end{cases}
\]
there exist $m,d>0$ such that
\[
\|\mf{u}_{\beta_n}\|_{L^\infty(\Omega)} \le m \qquad \text{and} \qquad \max_i \sup_{0<s \le m} \left|\frac{f_{i,\beta_n}(x,s)}{s}\right| \le d,
\]
and there exist functions $f_i \in \mathcal{C}^1(\Omega \times [0,m])$ such that
\[
f_{i,\beta_n} \to f_{i} \qquad \text{in $\mathcal{C}_{\loc}(\Omega \times [0,m])$ as $n \to \infty$}.
\]
Moreover, we are assuming that
\[
L_n:= \sup_{i=1,\dots,k} \sup_{x \in B_2} |\nabla (\eta u_{i,\beta_n})| \to +\infty.
\]
By \eqref{be assumpt} the quantities $\tilde r$ and $\tilde C$ in Proposition \ref{prp: monotonia almgren} can be chosen independently of $n$, and the following holds.

\begin{proposition}\label{prop: monotonia successione}
There exists $\tilde r,\tilde C>0$, depending only on $m$ and on $d$, such that for every $n$ the functions
\[
r \mapsto ( N(\mf{u}_{\beta_n},x_n,r) + 1 )e^{\tilde C r} \quad \text{and} \quad r \mapsto \left(\frac{1}{r^{N-1}}\int_{\partial B_r(x_n)} u_{i,\beta_n}^2\right) e^{\tilde Cr}
\]
are non-negative and monotone non-decreasing for $r \in (0,\tilde r]$.
\end{proposition}

We introduce the quantity
\begin{equation}\label{def R n}
R_{\beta_n} := \sup \left\{ r \in (0,\tilde r) :  ( N(\mf{u}_{\beta_n},x_n,r) + 1 )e^{\tilde C r} < 2-r \right\}.
\end{equation}
The role of $R_{\beta_n}$ will be clarified in the following. Before we establish some properties of the sequence $(R_{\beta_n})$. Firstly, since for any fixed $n$ all the components $u_{i,\beta_n}$ are positive, $N(\mf{u}_{\beta_n},x_n,0^+)=0$ for every $n$, and hence $R_{\beta_n} > 0$.

\begin{lemma}\label{lem: R_beta to 0}
$R_{\beta_n} \to 0$ as $n \to \infty$.
\end{lemma}

In order to prove the previous lemma, we need a result about uniform convergence.

\begin{lemma}\label{lem: unif convergence}
Let $g_n \in \C([0,1])$ be a sequence of monotone non-decreasing functions and let us assume that there exists $g \in \C([0,1])$ such that
$g_n \to g$ pointwise in $[0,1]$. Then $g_n \to g$ uniformly in $[0,1]$.
\end{lemma}
\begin{proof}
We shall introduce two functions based on the partition of the set $[0,1]$ in $k$ equal sub-intervals
 For $k \in \N$ fixed, we let
\[
    M_n^k(x) :=
\begin{cases}
    \sup_{j \geq n} (g_j(1/k), g(1/k))            &\text{if $x = 0$}\\
    \sup_{j \geq n} (g_j(l/k), g(l/k))  &\text{if $(l-1)/k < x \leq l/k$},
\end{cases}
\]
and
\[
    m_n^k(x) :=
\begin{cases}
    \inf_{j \geq n} (g_j((l-1)/k), g((l-1)/k))    &\text{if $(l-1)/k \leq x < l/k$}\\
    \inf_{j \geq n} (g_j(1-1/k), g(1-1/k))  &\text{if $x = 1$}.
\end{cases}
\]
From the monotonicity of the functions involved, we immediately obtain that
\[
    m_n^k(x) \leq g_j(x) \leq M_n^k(x) \quad \text{and} \quad m_n^k(x) \leq g(x) \leq M_n^k(x)  \qquad \text{for every }x \in [0,1], k \in \N, j \geq n
\]
and
\[
    |g_j(x) - g(x) | \leq M_n^k(x) - m_n^k(x) \qquad \text{for every }x \in [0,1], k \in \N, j \geq n.
\]
On the other hand, by pointwise convergence
for each $k\in\N$ there exists $N = N(k)$ such that
\[
\sup_x |g_j(x)-g(x)| \le  \sup_x \left(M_N^k(x) - m_N^k(x) \right) \leq 2 \osc_{k} g \qquad \text{for $j \geq N(k)$},
\]
where $\osc_{k} g$ is the maximal oscillation of $g$ in each sub-interval of the considered $k$-partition. Since $g$ is uniformly continuous in $[0,1]$, the thesis follows by taking the limit in $k$.
\end{proof}

We are now in a position to prove Lemma \ref{lem: R_beta to 0}.

\begin{proof}[Proof of Lemma \ref{lem: R_beta to 0}]
We have already observed in Remark \ref{rem: x_n to free boundary} that up to a subsequence $\{\mf{u}_{\beta_n}\}$ converges in $\mathcal{C}(\overline{B_2}) \cap H^1(B_2)$, as $n \to \infty$, to a limiting profile $\mf{u}_\infty \in \mathcal{G}(B_2)$, and also that $x_n \to x_\infty \in \overline{B_2} \cap \{\mf{u}_\infty=\mf{0}\}$. By the convergence
\[
r \mapsto ( N(\mf{u}_{\infty},x_{\infty},r) + 1 )e^{\tilde C r} \quad \text{is monotone non-decreasing};
\]
moreover, by Theorem \ref{prop: monot segregated}, $N(\mf{u}_\infty,x_\infty,0^+) \ge 1$. Let us assume by contradiction that
\[
    \limsup_{n \to \infty} R_{\beta_n} = R_\infty > 0.
\]
In light of the pointwise limit
\[
\lim_{n \to \infty} ( N(\mf{u}_{\beta_n},x_{n},r) + 1 )e^{\tilde C r} = ( N(\mf{u}_{\infty},x_{\infty},r) + 1 )e^{\tilde C r}
\]
valid for any $r \in (0,\tilde r)$, and of the monotonicity of the involved functions, we can apply Lemma \ref{lem: unif convergence} to obtain that up to a subsequence
\begin{align*}
 2 &> 2 - R_\infty = \lim_{n \to \infty} 2-R_{\beta_n} \geq \lim_{n \to \infty} ( N(\mf{u}_{\beta_n},x_{n},R_{\beta_n}) + 1 )e^{\tilde C R_{\beta_n}}  \\
& = ( N(\mf{u}_{\infty},x_{\infty},R_\infty) + 1 )e^{\tilde C R_\infty}\geq ( N(\mf{u}_{\infty},x_{\infty},0^+) + 1 ) \ge 2
\end{align*}
a contradiction. Here we used the fact that $x_\infty \in \{\mf{u}_\infty=\mf{0}\}$, so that by Theorem \ref{prop: monot segregated} we have $N(\mf{u}_\infty,x_\infty,0^+) \ge 1$.
\end{proof}

The previous results can be translated in terms of the elements of the blow-up sequence $\{\mf{v}_n\}$ (for the reader's convenience, we recall that $\{\mf{v}_n\}$ has been defined in \eqref{def blow-up}).

\begin{lemma}\label{corol: almgren v}
Let $\tilde r,\tilde C$ be defined in Proposition \ref{prop: monotonia successione}. For every $n \in \N$, the functions
\[
r \mapsto ( N(\mf{v}_n,0,r) + 1 )e^{\tilde C r_n r} \quad \text{and} \quad r \mapsto \left(\frac{1}{r^{N-1}}\int_{\partial B_r} v_{i,n}^2\right) e^{\tilde C r_n r}
\]
are non-negative and monotone non-decreasing for $r \in (0,\tilde r/r_n]$.
\end{lemma}

For the proof it is sufficient to check that
\begin{equation}\label{scaled quantities}
\begin{split}
\bullet & \quad E(\mf{v}_n,0,r) = \frac{\eta^2(x_n)}{L_n^2 r_n^2} E(\mf{u}_{\beta_n},x_n,r_n r) \\
\bullet & \quad H(\mf{v}_n,0,r) = \frac{\eta^2(x_n)}{L_n^2 r_n^2} H(\mf{u}_{\beta_n},x_n,r_n r) \\
\bullet & \quad N(\mf{v}_n,0,r) = N(\mf{u}_{\beta_n},x_n,r_n r) \\
\end{split}
\end{equation}
for every $0<r\le \tilde r/r_n$.

%

In the following lemma we enforce the conclusion of Proposition \ref{prop: riassunto blow-up}, showing that not only the limiting profile $\mf{v}$ of the blow-up sequence has at least one and at most two non-trivial components $v_1$ and $v_2$, but that both $v_1$ and $v_2$ are non-trivial and non-constant in the ball $B_2$.

%

\begin{lemma}\label{lem: 2 non deg}
There exists $C>0$ independent of $n$ such that
\[
\frac{1}{r^{N-1}} \int_{\pa B_r} v_{i,n}^2 \ge C
\]
for every $r \in [2, \tilde r/r_n]$ and $i=1,2$. In particular, both $v_{1,n}$ and $v_{2,n}$ are non-trivial and non-constant in $B_r$ for every $r \in  [2, \tilde r/r_n]$.
\end{lemma}
\begin{proof}
By Lemma \ref{corol: almgren v}, we infer that for $r \in [2, \tilde r/r_n]$ it holds
\[
\frac{1}{r^{N-1}}\int_{\pa B_r} v_{i,n}^2 \ge \left(\frac{1}{2^{N-1}}\int_{\pa B_2} v_{i,n}^2\right) e^{\tilde C r_n(2-\tilde r/r_n)},
\]
so that it is sufficient to show that there exists $C>0$ independent of $n$ such that
\begin{equation}\label{int su pa B_2}
\frac{1}{2^{N-1}} \int_{\pa B_2} v_{i,n}^2 \ge C \qquad \text{for $i=1,2$}
\end{equation}
We separate the prove according to whether $(M_n)$ is bounded or not.\\
\textbf{Case $i$)} \emph{$(M_n)$ is bounded.} \\
Assume by contradiction that the \eqref{int su pa B_2} is not true. By Proposition \ref{prop: riassunto blow-up}, $\mf{v}_n \to \mf{v}$ and $(v_1,v_2)$ solves \eqref{eq intera}. Moreover, $v_1(0)>0$ and, by subharmonicity, this implies that there exists $C>0$ such that
\[
\frac{1}{2^{N-1}} \int_{\pa B_2} v_1^2 \ge 2C \quad \Longrightarrow \quad \frac{1}{2^{N-1}} \int_{\pa B_2} v_{1,n}^2 \ge C
\]
for $n$ sufficiently large. As a consequence, using the subharmonicity of $v_2$
\[
\int_{\pa B_2} v_{2,n}^2 \to 0 \quad \Longrightarrow \quad v_2 \equiv 0 \quad \text{in $B_2$}.
\]
By the strong maximum principle, this means that $v_2 \equiv 0$ in $\R^N$, and thus $v_1$ is an entire, harmonic, non-constant and positive function, a contradiction.\\
\textbf{Case $ii$)} \emph{$M_n \to +\infty$.} \\
Arguing as in the first step, we deduce that $v_2 \equiv 0$ in $B_2$, and by Proposition \ref{prop: limiting profile} we know that $\mf{v} \in \mathcal{G}_{\loc}(\R^N)$. By the unique continuation property given by the Almgren monotonicity formula (see Remark \ref{rem: on unique continuation}) this implies that $v_1 > 0$ in $B_2$, and as a consequence $v_1$ is harmonic therein. To sum up, $v_1$ is a positive harmonic function in $B_2$ such that $v_1(0)=1$ and $|\nabla v_1(0)|=1$. Let $\theta \in \S^{N-1}$ be such that $\pa_\theta v_1=-1$. Since $\pa_\theta v_1$ is in turn harmonic in $B_2$, by the minimum principle $\inf_{B_r} \pa_\theta v_1 \le -1$ for any $r \in (0,2)$, and this immediately implies that $v_1$ changes sign in $B_2$, a contradiction.
\end{proof}

We introduce the counterpart of the radius $R_{\beta_n}$ in the blow-up setting, as
\begin{equation}\label{def bar r n}
\bar r_n:=\frac{R_{\beta_n}}{r_n}= \sup \left\{ r \in \left(0,\frac{\tilde r}{r_n}\right) : \left( N(\mf{v}_n,0,r)+1\right) e^{\tilde C r_n r} < 2-r_n r\right\}.
\end{equation}
By Lemma \ref{lem: R_beta to 0}, we deduce that $r_n \bar r_n \to 0$ as $n \to \infty$.

The value $\bar r_n$ will play a crucial role in the forthcoming argument. The idea is the following: for $r < \bar r_n$ it results $N(\mf{v}_n,0,r) \le 1$, and we shall show that consequently $\mf{v}_n$ exhibits a linear behaviour; on the contrary if $r >\bar r_n$, then $N(\mf{v}_n,0,r) \gtrapprox 1$, so that $\mf{v}_n$ is morally superlinear. We will prove, in the superlinear range $(\bar r_n,\tilde r/r_n)$ the function $(E(\mf{v}_n,0,r)+H(\mf{v}_n,0,r))/r^2$ is almost non-decreasing, uniformly in $n$. If the sequence $(\bar r_n)$ is bounded from above, this easily leads to a contradiction with Proposition \ref{prop: riassunto blow-up}. A more delicate situation takes place when $\bar r_n \to +\infty$, that is, $\bar r_n$ is an intermediate scale between the microscopic setting $r \le R<+\infty$ and the macroscopic scale $\tilde r/r_n \to +\infty$. In such a situation the function $\mf{v}_n$ transits from the linear behaviour to the superlinear one at the threshold $\bar r_n \to +\infty$, with $\bar r_n r_n \to 0$. In the linear range $[2,\bar r_n]$, we shall derive a uniform-in-$n$ perturbed version of the Alt-Caffarelli-Friedman monotonicity formula. In the superlinear range, the almost monotonicity of the function $(E(\mf{v}_n,0,r)+H(\mf{v}_n,0,r))/r^2$ holds. A delicate part of the proof consist in showing that a suitable combination of these results, which considered separately do not lead to any conclusive argument, permits to reach a contradiction also in this case.


In the next lemma we prove that the function $(E(\mf{v}_n,0,r)+H(\mf{v}_n,0,r))/r^2$ is almost monotone beyond the threshold $\bar r_n$. Note that by the definition of $\tilde r$ the numerator is positive in the whole interval $[0,\tilde r/r_n]$. Let us introduce
\[
    \varphi_{n}(r):=  2 \int_{\bar r_{n}}^r \left(2\frac{ e^{-\tilde Cr_n t} -1 }{t} - \frac{1}{t} r_n \bar r_n e^{-\tilde C r_n t} \right)\de t.
\]

\begin{remark}\label{rem: unif. bound phi_n}
Not only $\varphi_n$ is well defined for $r \in [\bar r_n,\tilde r/r_n]$, but it is bounded independently of $n$ in such interval:
\begin{equation}\label{stima sul resto}
\begin{split}
|\varphi_n(r)| &\le 4 \int_{\bar r_n}^{\tilde r/r_n} \frac{1- e^{-\tilde Cr_n t} }{t} \, \de t + 2r_n \bar r_n \int_{\bar r_n}^{\tilde r/r_n} \frac{  e^{-\tilde C r_n t}}{t} \, \de t \\ &\le 4\int_{\bar r_n}^{\tilde r/r_n} \tilde C r_n \, \de t + 2 \bar r_n r_n \int_{\bar r_n}^{\tilde r/r_n}  \frac{\de t}{t}  \le 4 \tilde C \tilde r + 2 \bar r_n r_n(|\log \tilde r|+|\log(\bar r_n r_n)|) \le C,
\end{split}
\end{equation}
with $C$ independent of $n$.
\end{remark}

\begin{lemma}\label{lem: Almgren ACF}
For every $n$, the function
\[
r \mapsto \frac{E(\mf{v}_n,0,r)+H(\mf{v}_n,0,r)}{r^2}e^{\tilde C r_n r-\varphi_n(r)} \qquad \text{is monotone non-decreasing for }r \in \left[\bar r_n,\frac{\tilde r}{r_n}\right].
\]
\end{lemma}
\begin{proof}
If $r \in [\bar r_n, \tilde r/r_n]$, then by the Almgren monotonicity formula
\[
N(\mf{v}_{n},0,r) -1 \ge 2\left( e^{-\tilde C r_n r} -1 \right) - r_n \bar r_n e^{-\tilde C r_n r}.
\]
As a consequence, recalling the expression \eqref{der di H} of the derivative of $H$, we have
\[
    \frac{\de}{\de r} \log \left(\frac{H(\mf{v}_{n}, 0,r)}{r^2}\right)  = \frac{2}{r} (N(\mf{v}_{n},0,r) -1) \ge \frac{4}{r}\left( e^{-\tilde C r_n r} -1 \right) - \frac{2}{r} r_n \bar r_n e^{-\tilde Cr_n r}.
\]
By integrating, we deduce that the function
\begin{equation}\label{monot H}
r \mapsto \frac{H(\mf{v}_{n}, 0,r)}{r^2} e^{-\varphi_{n}(r)} \quad \text{is monotone non-decreasing for } r \in \left[\bar r_{n}, \tilde r/r_n \right].
\end{equation}
To conclude, it is sufficient to observe that by Lemma \ref{corol: almgren v}
\begin{multline*}
\frac{\de}{\de r} \log \left( \frac{E(\mf{v}_n,0,r)+H(\mf{v}_n,0,r)}{r^2}e^{\tilde C r_n r-\varphi_n(r)}\right) \\= \frac{\de}{\de r} \log \left( (N(\mf{v}_n,0,r)+1)e^{\tilde C r_n r}\right)  + \frac{\de}{\de r} \log \left( \frac{H(\mf{v}_n,0,r)}{r^2}e^{-\varphi_n(r)} \right) \ge 0,
\end{multline*}
for $\bar r_n \le r \le \tilde r/r_n$.
\end{proof}

As a first consequence we can show that $(\bar r_n)$ cannot be bounded.
\begin{lemma}\label{raggio illimtato}
It holds $\bar r_n \to +\infty$ as $n \to \infty$.
\end{lemma}
\begin{proof}
Suppose, by contradiction, that up to a subsequence $\bar r_n \leq \bar r$ for some $\bar r > 0$. By the convergence of $\mf{v}_n \to \mf{v}$ in $\C_{\loc}(\R^N)$ and in $H^1_{\loc}(\R^N)$, and by Lemma \ref{lem: Almgren ACF} for any $r \in [\bar r+1, \tilde r/r_n)$
\begin{align*}
0 &\le \frac{E(\mf{v},0,r)+H(\mf{v},0,r)}{r^2} =\lim_{n \to \infty}\frac{E(\mf{v}_n,0,r)+H(\mf{v}_{n}, 0,r)}{r^2}  \\
&\le \lim_{n \to \infty} \sup_{r \in [\bar r_n, \tilde r/r_n]} r_n^2 \frac{E(\mf{v}_n,0,\tilde r/r_n)+H(\mf{v}_{n}, 0,\tilde r/r_n)}{\tilde r^2} e^{\tilde C (\tilde r-r_n r)+\varphi_n(r)-\varphi_n(\tilde r/r_n)} \\
&\le \lim_{n \to \infty} C r_n^2 \frac{E(\mf{v}_n,0,\tilde r/r_n)+H(\mf{v}_{n}, 0,\tilde r/r_n)}{\tilde r^2} = \lim_{n \to \infty} C \eta^2(x_n) \frac{E(\mf{u}_{\beta_n},x_n,\tilde r)+H(\mf{u}_{\beta_n}, x_n,\tilde r)}{L_n^2\tilde r^2},
\end{align*}
where we used the identities \eqref{scaled quantities} and the uniform boundedness of $\{\varphi_n\}$, see Remark \ref{rem: unif. bound phi_n}. Since both $E(\mf{u}_{\beta_n},x_n,\tilde r)$ and $H(\mf{u}_{\beta_n}, x_n,\tilde r)$ are also uniformly bounded (for the boundedness of $E(\mf{u}_{\beta_n},x_n,\tilde r)$, it is possible to proceed as in points (5) and (6) of Lemma \ref{lem: mega lemma}), while $L_n \to +\infty$, the last limit tends to $0$. As a consequence $\mf{v} \equiv 0$ in $B_r$, in contradiction with Lemma \ref{lem: 2 non deg}.
\end{proof}

Summing up, we have shown that if $L_n \to +\infty$ then necessarily
\[
    \bar r_n \to \infty \quad \text{while} \quad r_n \bar r_n \to 0 \qquad \text{as $n\to \infty$}.
\]
It remains to prove that also in this case we reach a contradiction with Lemma \ref{lem: 2 non deg}. To this end, let us introduce $J_n(r) :=r^{-4} J_{1,n}(r) \cdot J_{2,n}(r)$, where
\begin{align*}
J_{1,n}(r) &:= \int_{B_r} \left(\left|\nabla v_{1,n}\right|^2 + M_n a_{12} v_{1,n}^2 v_{2,n}^2-v_{1,n} f_{1,n}(x,v_{1,n}) \right)|x|^{2-N}  \\
J_{2,n}(r)& := \int_{B_r} \left(\left|\nabla v_{2,n}\right|^2 + M_n a_{12} v_{1,n}^2 v_{2,n}^2-v_{2,n} f_{2,n}(x,v_{2,n}) \right)|x|^{2-N}.
\end{align*}
A crucial step in the proof of Theorem \ref{thm: main general be} is the validity of the Alt-Caffarelli-Friedman monotonicity formula of Subsection \ref{sub: ACF} for $J_n$, uniformly in $n$.
\begin{lemma}\label{lem: ACF uniform}
There exists $C>0$ independent of $n$ such that $J_{1,n}(r) \ge C$ and $J_{2,n}(r) \ge C$ for every $r \in [2,\bar r_n/3]$, and
\[
r \mapsto J_n(r) e^{-C M_n^{-1/4} r^{-1/2} + C r_n^2 r^2} \quad \text{is monotone non-decreasing for $r \in [2,\bar r_n/3]$}.
\]
\end{lemma}
The proof consists in verifying that the assumptions ($h0$)-($h3$) of Theorem \ref{thm: ACF} are satisfied in the range $[2,\bar r_n/3]$, with constants uniform in $n$. In doing this, we shall strongly use the fact that $r \in (0,\bar r_n]$, the range where the function $\mf{v}_n$ has linear behaviour. Since the proof is quite long and a little bit technical, we postpone it in Subsection \ref{sub: ACF uniform}, and now we proceed with the conclusion of the proof of Theorem \ref{thm: main general be}.

\begin{proof}[Conclusion of the proof of Theorem \ref{thm: main general be}]
We aim at proving the validity of the following chain of inequalities, connecting the Alt-Caffarelli-Friedman monotonicity formula of Lemma \ref{lem: ACF uniform} to the Almgren-type monotonicity formula of Lemma \ref{lem: Almgren ACF}:
\begin{equation}\label{stima cruciale}
\begin{split}
C \le J_n(2) & \le C J_n\left(\frac{\bar r_n}{3}\right) \le C\left( \frac{E(\mf{v}_n,0,\bar r_n)+H(\mf{v}_n,0,\bar r_n)}{\bar r_n^2} + o_n(1)\right)^2 \\
&\le C\left( r_n^2\frac{E(\mf{v}_n,0,\tilde r/r_n)+H(\mf{v}_n,0,\tilde r/r_n)}{\tilde r^2} + o_n(1)\right)^2,
\end{split}
\end{equation}
where $o_n(1) \to 0$ as $n \to \infty$. Once that this is proved, the conclusion easily follows: indeed, as in Lemma \ref{raggio illimtato}, the last quantity tends to $0$ as $n \to +\infty$, in contradiction with the fact that $J_n(2) \ge C$. Hence we have to verify the validity of \eqref{stima cruciale}. The first two inequalities follow by Lemma \ref{lem: ACF uniform}, the last one can be proved as in Lemma \ref{raggio illimtato}, and therefore we have only to check that
\begin{equation}\label{stima cruciale 2}
J_n\left(\frac{\bar r_n}{3}\right) \le C\left(\frac{E(\mf{v}_n,0,\bar r_n)+H(\mf{v}_n,0,\bar r_n)}{\bar r_n^2} + o_n(1)\right)^2.
\end{equation}
We emphasize that, recalling the definition of $J_n(r)$ and $E(\mf{v}_n,0,r)$, this can be considered an inequality relating the \emph{geometric mean} with the \emph{arithmetic mean} of suitable energy functionals of $\mf{v}_n$ in $B_r$. By definition
\begin{equation*}
\begin{split}
J_n\left(\frac{\bar r_n}{3}\right) &= \left( \frac{9}{\bar r_n^2} \int_{B_{\bar r_n/3}} \left(\left|\nabla v_{1,n}\right|^2 + M_n a_{12} v_{1,n}^2 v_{2,n}^2-v_{1,n} f_{1,n}(x,v_{1,n}) \right)|x|^{2-N} \right) \\
& \qquad \cdot \left(\frac{9}{\bar r_n^2} \int_{B_{\bar r_n/3}} \left(\left|\nabla v_{2,n}\right|^2 + M_n a_{12} v_{1,n}^2 v_{2,n}^2-v_{2,n} f_{2,n}(x,v_{2,n}) \right)|x|^{2-N} \right).
\end{split}
\end{equation*}
We control both terms in the product on the right hand side in the same way, so here we consider only the first term. It holds
\begin{multline*}
\frac{1}{\bar r_n^2} \int_{B_{\bar r_n/3}} \left(\left|\nabla v_{1,n}\right|^2 + M_n a_{12} v_{1,n}^2 v_{2,n}^2-v_{1,n} f_{1,n}(x,v_{1,n}) \right)|x|^{2-N} \\
\le \frac{1}{\bar r_n^2} \int_{B_{\bar r_n}} \left(\left|\nabla v_{1,n}\right|^2 + M_n a_{12} v_{1,n}^2 v_{2,n}^2-v_{1,n} f_{1,n}(x,v_{1,n}) \right)|x|^{2-N}  + \frac{1}{\bar r_n^2}\int_{B_{\bar r_n} \setminus B_{\bar r_n/3}} \frac{v_{1,n} f_{1,n}(x,v_{1,n})}{|x|^{N-2}}.
\end{multline*}

First, we claim that
\begin{equation}\label{claim stime finali 1}
\lim_{n \to \infty} \left| \frac{1}{\bar r_n^2}\int_{B_{\bar r_n} \setminus B_{\bar r_n/3}} \frac{v_{1,n} f_{1,n}(x,v_{1,n})}{|x|^{N-2}}\right| = 0.
\end{equation}
Indeed by definition of $v_{i,n}$ and $f_{i,n}$, and by the assumption \eqref{be assumpt}, it results
\begin{align*}
\left|\frac{1}{\bar r_n^2}\int_{B_{\bar r_n} \setminus B_{\bar r_n/3}} \frac{v_{1,n} f_{1,n}(x,v_{1,n})}{|x|^{N-2}} \right|  & \le \frac{d r_n^2}{\bar r_n^2}\int_{B_{\bar r_n} \setminus B_{\bar r_n/3}} \frac{v_{1,n}^2}{|x|^{N-2}} \\
& =\frac{d \eta^2(x_n)}{L_n^2 \bar r_n^2 r_n^2} \int_{B_{\bar r_n r_n}(x_n) \setminus B_{\bar r_n r_n/3}(x_n)} \frac{u_{1,\beta_n}^2}{|x-x_n|^{N-2}} \\
& \le \frac{C m^2}{L_n^2 \bar r_n^N r_n^N} \int_{B_{\bar r_n r_n}(x_n) } 1 \le \frac{C}{L_n^2} \to 0
\end{align*}
as $n \to \infty$, where we used the uniform boundedness of $\{\mf{u}_n\}$ and the fact that $L_n \to +\infty$.

Secondly, we claim that
\begin{equation}\label{claim 2 stime finali}
\frac{1}{\bar r_n^2} \int_{B_{\bar r_n}} \left(\left|\nabla v_{1,n}\right|^2 + M_n a_{12} v_{1,n}^2 v_{2,n}^2-v_{1,n} f_{1,n}(x,v_{1,n}) \right)|x|^{2-N} \le C \frac{E(\mf{v}_n,0,\bar r_n)+H(\mf{v}_n,0,\bar r_n)}{\bar r_n^2}.
\end{equation}
To prove it, let us test the equation for $v_{i,n}$ against $v_{i,n} |x|^{2-N}$: integrating by parts as in the proof of Lemma \ref{lem: positivita}, we deduce that
\begin{equation}\label{stima finale 3}
\begin{split}
\int_{B_{\bar r_n}} & \left(\left|\nabla v_{1,n}\right|^2 + M_n a_{12} v_{1,n}^2 v_{2,n}^2-v_{1,n} f_{1,n}(x,v_{1,n}) \right)|x|^{2-N} \\
& \quad \le \frac{1}{\bar r_n^{N-2}} \int_{\pa B_{\bar r_n}} v_{1,n} \pa_{\nu} v_{1,n} +\frac{N-2}{2 \bar r_n^{N-1}} \int_{\pa B_{\bar r_n}} v_{1,n}^2 \\
& \quad = \frac{1}{\bar r_n^{N-2}} \int_{B_{\bar r_n}} \left(\left|\nabla v_{1,n}\right|^2 + M_n a_{12} v_{1,n}^2 v_{2,n}^2-v_{1,n} f_{1,n}(x,v_{1,n}) \right) + \frac{N-2}{2 \bar r_n^{N-1}} \int_{\pa B_{\bar r_n}} v_{1,n}^2.
\end{split}
\end{equation}
For every $i=1,\dots,k$ and for every $n$
\begin{multline*}
\frac{1}{\bar r_n^{N-2}} \int_{B_{\bar r_n}}  \left(\left|\nabla v_{i,n}\right|^2 + M_n a_{12} v_{i,n}^2 \sum_{j \neq i} v_{j,n}^2-v_{i,n} f_{i,n}(x,v_{i,n}) \right) + \frac{N-2}{2 \bar r_n^{N-1}} \int_{\pa B_{\bar r_n}} v_{i,n}^2 \\
 \ge \frac{\eta^2(x_n)}{L_n^2 r_n^2}\left( \frac{1}{(\bar r_n r_n)^{N-2}} \int_{B_{\bar r_n r_n}(x_n)} |\nabla u_{i,\beta_n}|^2 - \frac{d (\bar r_n r_n)^2}{(\bar r_n r_n)^N} \int_{B_{\bar r_n r_n}(x_n)} u_{i,\beta_n}^2 \right) \\
+\frac{\eta^2(x_n)}{L_n^2 r_n^2} \frac{N-2}{2 (\bar r_n r_n)^{N-1}} \int_{\pa B_{\bar r_n r_n}} u_{i,\beta_n}^2 \\
\ge \frac{\eta^2(x_n)}{L_n^2 r_n^2} \frac{1}{(\bar r_n r_n)^{N-2}} \left( 1-\frac{d (\bar r_n r_n)^2}{N-1}\right) \int_{B_{ \bar r_n r_n}(x_n)} |\nabla u_{i,\beta_n}|^2 \\
 + \frac{\eta^2(x_n)}{L_n^2 r_n^2} \left( \frac{N-2}{2}-\frac{d (\bar r_n r_n)^2}{N-1}\right) \frac{1}{(\bar r_n r_n)^{N-1}} \int_{\pa B_{\bar r_n r_n}(x_n)} u_{i,\beta_n}^2 \ge 0,
\end{multline*}
where we used the fact that $\bar r_n r_n \to 0$ and the Poincar\'e inequality (Lemma \ref{lem: poincare}). Coming back to \eqref{stima finale 3}, we deduce that
\begin{align*}
\int_{B_{\bar r_n}} & \left(\left|\nabla v_{1,n}\right|^2 + M_n a_{12} v_{1,n}^2 v_{2,n}^2-v_{1,n} f_{1,n}(x,v_{1,n}) \right)|x|^{2-N}  \\
& \le  \sum_{i=1}^k \left[\frac{1}{\bar r_n^{N-2}} \int_{B_{\bar r_n}} \left(\left|\nabla v_{i,n}\right|^2 + M_n a_{12} v_{i,n}^2 \sum_{j \neq i} v_{j,n}^2-v_{i,n} f_{i,n}(x,v_{i,n}) \right) + \frac{N-2}{2 \bar r_n^{N-1}} \int_{\pa B_{\bar r_n}} v_{i,n}^2\right] \\
& = E(\mf{v}_n,0,\bar r_n)+ \frac{N-2}{2}H(\mf{v}_n,0,\bar r_n) \le C \left( E(\mf{v}_n,0,\bar r_n)+ H(\mf{v}_n,0,\bar r_n) \right).
\end{align*}
Multiplying the first and the last term by $\bar r_n^{-2}$, the claim \eqref{claim 2 stime finali} follows.

At this point it is sufficient to observe that claims \eqref{claim stime finali 1} and \eqref{claim 2 stime finali} imply that \eqref{stima cruciale 2} holds, which completes the proof.
\end{proof}

\subsection{Proof of Lemma \ref{lem: ACF uniform}}\label{sub: ACF uniform}

We will often use the fact that the function $N(\mf{v}_n,0,r)$ can be controlled from below and from above by positive constants in the range $[2,\bar r_n]$.

\begin{lemma}\label{lem: N limitato da sopra}
There exists $\sigma \in (0,1)$ such that $\sigma \le N(\mf{v}_n,0,r) \le 1$ for every $r \in [2,\bar r_n]$, for every $n$. As a consequence
\begin{align*}
& r \mapsto \frac{H(\mf{v}_n,0,r)}{r^2} \quad \text{is monotone non-increasing for $r \in [2,\bar r_n]$} \\
& r \mapsto \frac{H(\mf{v}_n,0,r)}{r^{2\sigma}} \quad \text{is monotone non-decreasing for $r \in [2,\bar r_n]$}.
\end{align*}
\end{lemma}
\begin{proof}
By the Almgren monotonicity formula
\[
(N(\mf{v}_n,0,r) +1)e^{\tilde C r_n r} \le (N(\mf{v}_n,0,\bar r_n) +1)e^{\tilde C r_n \bar r_n} = 2-r_n \bar r_n
\]
for every $r \in [0,\bar r_n]$. This gives the desired upper bound on $N$. For the lower bound, using again the Almgren monotonicity formula we have
\[
\left(N(\mf{v}_n,0,r)+1\right)e^{\tilde C r_n r} \ge \left(N(\mf{v}_n,0,2)+1\right)e^{\tilde C r_n 2}
\]
for every $r \in [0,\bar r_n]$, which readily implies
\begin{equation}\label{stima su N dal basso}
N(\mf{v}_n,0,r) \ge  (N(\mf{v}_n,0,2)+1)e^{-\tilde C r_n r} -1.
\end{equation}
Now, as by Proposition \ref{prop: riassunto blow-up} and Lemma \ref{lem: 2 non deg} both $v_{1,n}$ and $v_{2,n}$ are non-trivial and non-constant in $B_2$, we have
\[
\int_{B_2} \left(\sum_i \left|\nabla v_{i,n}\right|^2 + 2 M_n \sum_{i<j} a_{ij} v_{i,n}^2 v_{j,n}^2\right)\ge C>0
\]
Since $f_{i,n} \to 0$ uniformly in $\Omega_n$ (see point (1) of Lemma \ref{lem: mega lemma}), we deduce that $E(\mf{v}_n,0,2) \ge C >0$, and by uniform convergence $\mf{v}_n \to \mf{v}$ we infer that $N(\mf{v}_n,0,2) \ge C >0$ independently of $n$. Since $r_n r \le r_n \bar r_n \to 0$ as $n \to \infty$ for every $r \le \bar r_n$, coming back to the estimate \eqref{stima su N dal basso} we conclude that there exists $C>0$ such that
\[
N(\mf{v}_n,0,r) \ge (1+C)e^{-\tilde C r_n r}-1 \ge \sigma >0  \qquad \text{for every $r \in [2,\bar r_n]$},
\]
for every $n$ sufficiently large. The second part of the thesis is now a direct consequence of Lemma \ref{lem: doubling}.
\end{proof}

In the next lemma we make rigorous the concept that $\mf{v}_n$ behaves in a linear way up to the threshold $\bar r_n$.

\begin{lemma}\label{lem: blow-down}
Let $(\rho_n)$ be any sequence such that $\rho_n \to \infty$ and $\rho_n \leq \bar r_n / 3$. Then there exist $\gamma > 0$ and $1 \leq h < l \leq k$ such that, up to a subsequence, the blow-down sequence
\[
    \tilde v_{i,n}(x) := \frac{v_{i,n}(\rho_n x)}{\sqrt{H(\mf{v}_n, 0, \rho_n)}}
\]
converges in $H^1(B_1) \cap \C (\overline{B_1})$, up to a rotation, to the function $\tilde{\mf{v}}$ defined by
\[
    \tilde v_h(x)=\gamma x_1^+ \qquad \tilde v_l(x) = \gamma x_1^- \qquad \tilde v_j(x) = 0 \qquad \text{for every $j \neq h,l$}.
\]
\end{lemma}
\begin{proof}
We start with the observation that each $\tilde{\mf{v}}_n$ solves the system
\[
    -\Delta \tilde v_{i,n} =  \frac{ \rho_n^2}{\sqrt{H(\mf{v}_n, 0, \rho_n)}} f_{i,n}\left(\rho_n,v_{i,n} (\rho_n x ) \right)- \rho_n^2 H(\mf{v}_n,0,\rho_n) M_n \tilde v_{i,n} \sum_{j \neq i} a_{ij} \tilde v_{j,n}^2
\]
in a set $\tilde \Omega_n \supset B_3$ (this follows directly by the fact that $\Omega_n \supset B_{1/r_n}$). Since $\rho_n \to +\infty$, $H(\mf{v}_n,0,\rho_n) \ge C$ (by Lemma \ref{lem: 2 non deg}) and $M_n \ge C >0$, we infer that the new competition parameter $\rho_n^2 H(\mf{v}_n,0,\rho_n) M_n \to +\infty$. Furthermore, recalling assumption \eqref{be assumpt} and the definition of $f_{i,n}$, we have
\begin{equation}\label{elliptic ineq}
    -\Delta \tilde v_{i,n} \le  \frac{ \rho_n^2}{\sqrt{H(\mf{v}_n, 0, \rho_n)}} f_{i,n}\left(\rho_n,v_{i,n} (\rho_n x ) \right) \le d \left(\rho_n r_n\right)^2 \tilde v_{i,n}
\end{equation}
in $B_3$. We wish to deduce that 
$\{\tilde{\mf{v}}_n\}$ is uniformly bounded in $B_2$. If $\{\mf{v}_n\}$ is uniformly bounded in $H^1(B_3)$ and the coefficients on the right hand side are uniformly bounded, this follows by a classical Brezis-Kato argument. The boundedness of the coefficients is given by $\rho_n r_n \le \bar r_n r_n \to 0$ as $n \to \infty$, as shown in Lemma \ref{lem: R_beta to 0}. Thus it remains to show that $\{\tilde{\mf{v}}_n\}$ is uniformly bounded in $H^1(B_3)$. By Lemma \ref{lem: N limitato da sopra}
\[
N(\tilde{\mf{v}}_n,0,\rho) \leq N\left(\mf{v}_n,0, \rho \rho_n \right) \le 1
\]
for every $0 \le \rho \le 3$, so that by Lemma \ref{lem: doubling}
\[
H(\tilde{\mf{v}}_n,0,\rho) = \frac{H(\mf{v}_n,0,\rho_n \rho)}{H(\mf{v}_n,0,\rho_n)} \le \rho^2
\]
for every $1 \le \rho \le 3$. Therefore
\[
E(\tilde{\mf{v}}_n,0,3) = N(\tilde{\mf{v}}_n,0,3)H(\tilde{\mf{v}}_n,0,3) \le 9.
\]
It is easy to check that this gives the desired upper bound: indeed thanks to the Poincar\`e inequality in Lemma \ref{lem: poincare} we have
\begin{align*}
E(\tilde{\mf{v}}_n,0,3) &\ge \frac{1}{3^{N-2}} \int_{\pa B_3} \sum_i |\nabla \tilde v_{i,n}|^2 - \frac{(\rho_n r_n)^2}{3^N} \int_{B_r} \sum_i \tilde v_{i,n}^2 \\
& \ge \frac{1}{3^{N-2}}\left(1-(\rho_n r_n)^2\right) \int_{B_3} \sum_i |\nabla \tilde v_{i,n}|^2 - (\rho_n r_n)^2 H(\tilde{\mf{v}}_n,0,3)\\
& \ge C \int_{B_3} \sum_i |\nabla \tilde v_{i,n}|^2 -o_n(1),
\end{align*}
with $o_n(1) \to 0$ as $n \to \infty$.
Coming back to \eqref{elliptic ineq}, we deduce that the sequence $\{\tilde{\mf{v}}_n\}$ is bounded in $L^\infty(B_2)$, and in light of the local version of the main results in \cite{nttv}, see \cite{stz}, we conclude that up to a subsequence $\tilde{\mf{v}}_n \to \tilde{\mf{v}}$ in $\mathcal{C}(B_{3/2})$ and in $H^1(B_{3/2})$ as $n \to \infty$. Since
\[
    \left| \frac{ \rho_n^2}{\sqrt{H(\mf{v}_n, 0, \rho_n)}} f_{i,n}\left(\rho_n,v_{i,n} (\rho_n x ) \right)\right| \le (\rho_n r_n )^2 \|\tilde v_{i,n}\|_{L^\infty(B_2)} \to 0
 \]
as $n \to \infty$, by Proposition \ref{prop: limiting profile} the limiting profile $\tilde{\mf{v}}$ belongs to the class $\mathcal{G}(B_{3/2})$, with
\[
    \Delta \tilde v_i = 0 \qquad \text{in $\{\tilde v_i>0\}$}.
\]
Moreover, $0 \in \{\tilde{\mf{v}}=0\}$, the free-boundary of $\tilde{\mf{v}}$, since the sequence $(\mf{v}_n(0))$ is bounded in $0$ while by Lemma \ref{lem: N limitato da sopra}
\[
H(\mf{v}_n,0,\rho_n) \ge \frac{H(\mf{v}_n,0,2)}{4^\sigma} r^{2\sigma}\to +\infty \qquad \text{as $n \to \infty$}.
\]
By Proposition \ref{prop: monot segregated senza f} we conclude that
\[
    1\le N(\tilde{\mf{v}},0,0^+) \le N(\tilde{\mf{v}},0,r) \le 1,
\]
that is $N(\tilde{\mf{v}},0,r)=1$, for every $r \in (0,3/2]$, and hence $\tilde{\mf{v}}$ is homogeneous of degree $1$. Moreover, thanks to Theorem \ref{prop: on multipl} the occurrence of self-segregation phenomena is not allowed, so that up to a rotation there exists $h,j \in \{1,\dots,k\}$ such that
\[
\tilde v_h(x)=\gamma x_1^+ \qquad \tilde v_l(x) = \gamma x_1^- \qquad \tilde v_j(x) = 0 \qquad \text{for every $j \neq h,l$},
\]
where $\gamma>0$ is uniquely determined by the normalization condition $H(\tilde{\mf{v}},0,1)=1$.
\end{proof}
\begin{remark}
The fact that $N(\mf{v}_n,0,r) \le 1$ below the threshold $\bar r_n$ implies that suitable scaled sequences $\{\tilde{\mf{v}}_n\}$ are converging to explicit segregated profile of linear type, as in the thesis of Lemma \ref{lem: blow-down}. If we exceed the threshold $\bar r_n$ in general $N(\mf{v}_n,0,r) > 1$, and the characterization of the possible blow-up limits becomes more involved and remains still an open problem.
\end{remark}

The next step consists in showing that, up to the scale $\bar r_n$, there is a balance between $v_{1,n}$ and $v_{2,n}$, while the other components are of smaller order. Before, we need a technical result.

\begin{definition}
We denote as $\Sigma_{2,k}$ the subset of $\R^k$ of points with at most two non-trivial components, that is
\[
    \Sigma_{2,k} := \{\mf{x} \in \R^k: \ \exists i, j \in \{1,\dots, k\} \text{ such that } x_h = 0 \ \forall h \neq i, j\}.
\]
For every $\mf{x} \in \R^k$ and $g \in \C([0,1]; \R^k)$, we let
\[
    \dist(\mf{x}, \Sigma_{2,k}) = \inf_{\mf{y} \in \Sigma_{2,k}} \|\mf{x} - \mf{y}\|, \quad \dist(g([0,1]), \Sigma_{2,k}) = \sup_{x \in [0,1]} \dist(f(x), \Sigma_{2,k}).
\]
\end{definition}

\begin{lemma}\label{lem: elementare Watson}
Let $\eps \in (0,1)$ be a fixed number and let $g_n \in \C( [0,1], \R^k)$ be a sequence of continuous functions such that
\[
    g_n([0,1]) \subset E_{\eps} := \left\{ \mf{x} = (x_1, \dots, x_k) \in \R^k: x_i \geq 0,\  x_i \leq 1-\eps,\ \sum_{i=1}^{k} x_i = 1 \right\}.
\]
If the limit
\[
    \lim_{n \to \infty} \dist(g_n([0,1]), \Sigma_{2,k} ) = 0
\]
holds true, then (up to a subsequence) there exist $i \neq j$ such that for $n$ sufficiently large
\[
    \frac{\eps}{2} < g_{i,n}(x) ,  g_{j,n}(x) < 1-\frac{\eps}{2} \quad \text{and} \quad  g_{h,n} \to 0 \quad \text{uniformly in $[0,1]$ for $h \neq i, j$}.
\]
\end{lemma}
\begin{proof}
The set $E_\eps \cap \Sigma_{2,k}$ is made of $k (k-1)/2$ connected components, given by the reflections of the set
\[
    A = \{ x_1 + x_2 = 1, \ \eps \leq x_1 \leq 1-\eps, \ \eps \leq x_2 \leq 1-\eps,\  x_i = 0 \ \forall i \geq 3 \}.
\]
For the assumptions we evince that for any $\delta > 0$, there exists $\bar n$ sufficiently large such that
\[
    g_n([0,1]) \subset (E_\eps \cap \Sigma_{2,k}) + B_\delta \qquad \text{for all $n \geq \bar n$}.
\]
Evidently, for $\delta$ small enough, the set $(E_\eps \cap \Sigma_{2,k}) + B_\delta$ is again made of $k (k-1)/2$ disjoint components. It follows then that for a suitable subsequence there exists a connected component of $E_\eps \cap \Sigma_{2,k}$, say the set $A$, such that
\[
    \lim_{n \to \infty} \dist(g_n([0,1]), A) = 0.
\]
The thesis is then an immediate consequence.
\end{proof}

\begin{lemma}\label{lem: controllo rapporti}
There exists $\lambda>0$ independent of $n$ such that
\[
\frac{1}{\lambda} \le \frac{\int_{\pa B_r} v_{1,n}^2  }{\int_{\pa B_r} v_{2,n}^2  } \le \lambda
\]
for every $2 \le r \le \bar r_n/3$, while on the contrary for $j=3,\dots,k$ we have
\[
\sup_{r \in [2,\bar r_n/3]} \frac{\int_{\pa B_r} v_{j,n}^2  }{\int_{\pa B_r} v_{1,n}^2  } \to 0 \qquad \text{as $n \to \infty$}.
\]
\end{lemma}

\begin{proof}
The proof is based on an application of Lemma \ref{lem: elementare Watson}: more precisely, given the sequence $\{\mf{v}_n\}$, let us introduce the family of auxiliary functions
\[
    g_{i,n}(\rho) :=
    \begin{cases}
        \displaystyle \frac{1}{H(\mf{v}_n, 0, \rho \bar r_n/3)} \frac{1}{(\rho \bar r_n/3)^{N-1}} \int_{\pa B_{\rho \bar r_n/3}} v_{i,n}^2 & \text{for $\frac{6}{\bar r_n} \leq \rho \leq 1$}\\
        \\
        \displaystyle \frac{1}{H(\mf{v}_n, 0, 2)} \frac{1}{2^{N-1}} \int_{\pa B_2} v_{i,n}^2              & \text{for $0 \leq \rho \leq \frac{6}{\bar r_n}$}.
    \end{cases}
\]
the thesis follows once we have shown that $\{g_n\}$ satisfies the assumptions of Lemma \ref{lem: elementare Watson}. By construction, we have that each $g_{i,n}$ is continuous, $g_{i,n} \geq 0$, and $\sum_{i=1}^{k} g_{i,n} (x) = 1$ for all $x \in [0,1]$.

\paragraph{\textbf{Step 1)} \emph{There exists $\eps \in (0,1)$ such that $g_{i,n} (x) \leq 1 - \eps$ for all $x \in [0,1]$, independently of $n$}}

By contradiction, let us assume that is there exist an index $i$ and a sequence $s_n \in [0,1]$ such that
\begin{equation}\label{cons abs}
    g_{i,n}(s_n) \to 1 \quad \text{and} \quad g_{j,n}(s_n) \to 0 \qquad \forall j \neq i.
\end{equation}
The local uniform convergence $\mf{v}_n \to \mf{v}$ and Lemma \ref{lem: 2 non deg} imply that necessarily $s_n \bar r_n \to +\infty$ as $n \to \infty$. Let us introduce a new sequence
\[
    \tilde v_{i,n}(x) := \frac{v_{i,n}(s_n \bar r_n x/3)}{\sqrt{H(\mf{v}_n, 0, s_n \bar r_n/3)}}.
\]
We are in a position to apply Lemma \ref{lem: blow-down} (with $\rho_n := s_n \bar r_n/3$) and to conclude that the uniform limit of $\{\tilde{\mf{v}}_n\}$ contains at least two non trivial components, in contradiction with \eqref{cons abs}.

\paragraph{\textbf{Step 2)} \emph{$\lim_{n \to \infty} \dist(g_n([0,1]), \Sigma_{2,k}) = 0$}}

As before, let us assume by contradiction, that there exist $\eps > 0$, three distinct indices $i, j ,k$, and a sequence $s_n \in (0,1)$ such that up to a subsequence
\[
    g_{i,n}(s_n) \geq \eps,\quad g_{j,n}(s_n) \geq \eps, \quad g_{k,n}(s_n) \geq \eps
\]
for any $n$ sufficiently large, and let us introduce again the sequence
\[
    \tilde v_{i,n}(x) := \frac{v_{i,n}(s_n \bar r_n x/3)}{\sqrt{H(\mf{v}_n, 0, s_n \bar r_n/3)}}.
\]
As a result of Lemma \ref{lem: blow-down}, the uniform limit of $\{\tilde{\mf{v}}_n\}$ contains at most two non trivial components, a contradiction.
\end{proof}

\begin{remark}
We point out that such proof rests upon Lemma \ref{lem: blow-down}, which reflects the linear behaviour of $\mf{v}_n$ in $B_r$ for $r \in (0,\bar r_n]$. We do not expect that in a superlinear range the same result holds.
\end{remark}

Lemmas \ref{lem: 2 non deg} and \ref{lem: controllo rapporti} imply that assumption ($h2$) in Theorem \ref{thm: ACF} is satisfied with constants $\mu,\lambda>0$ independent of $n$ in the interval $r \in (2,\bar r_n/3)$. We now consider assumptions ($h0$) and ($h3$).

\begin{lemma}
Provided $n$ is sufficiently large, there holds
\[
|f_{i,n}(x,v_{i,n})| \le d r_n^2 v_{i,n} \quad \text{in $\Omega_n$} \quad \text{with} \quad \frac{d r_n^2 \bar r_n^2}{9} < \left(\frac{N-2}{2}\right)^2
\]
for $i=1,2$. We recall that $\Omega_n$ is the domain of definition of $\mf{v}_n$.
\end{lemma}
\begin{proof}
It is an easy consequence of the definition of $f_{i,n}$, and of the fact that $r_n^2 \bar r_n^2 \to 0$ as $n \to \infty$, see Lemma \ref{lem: R_beta to 0}.
\end{proof}

At this point we can show that the quantities $\Lambda_{1,n}(r)$ and $\Lambda_{2,n}(r)$ are uniformly bounded from below by a positive constant up to the scale $\bar r_n$.

\begin{lemma}\label{lem: lambda_i non zero}
There exists $C>0$ independent of $n$ such that
\[
\Lambda_{1,n}(r),\Lambda_{2,n}(r) \ge C \qquad \text{for } r \in \left[2,\frac{\bar r_n}{3}\right].
\]
\end{lemma}

\begin{proof}
By contradiction, we assume that there exists $\rho_n \in [2,\bar r_n/3]$ such that (without loss of generality) $\lim_n \Lambda_{1,n}(\rho_n) \le 0$, that is
\[
\lim_{n \to \infty} \frac{\rho_n^2\int_{\pa B_{\rho_n}} |\nabla_\theta v_{1,n}|^2 + M_n a_{12} v_{1,n}^2 v_{2,n}^2-v_{1,n} f_{1,n}(x,v_{1,n})  }{\int_{\pa B_{\rho_n}} v_{1,n}^2} \le  0.
\]
Either $\rho_n \to +\infty$ or $(\rho_n)$ is bounded. In the former case, we consider the scaled sequence
\[
\tilde v_{i,n}(x):= \frac{v_{i,n}(\rho_n x)}{\sqrt{H(\mf{v}_n,0,\rho_n)}},
\]
which is well defined in $B_{3}$ since $\rho_n \le \bar r_n/3$. The asymptotic behaviour of this blow-down sequence is again contained in Lemma \ref{lem: blow-down}, from which we know that $\tilde{\mf{v}}_n \to \tilde{\mf{v}}$ such that (up to a rotation)
\[
\tilde v_{i}=\gamma x_1^+, \quad \tilde v_{j}= \gamma x_1^-, \quad \tilde v_k = 0 \quad \text{for every $k \neq i,j$},
\]
for a suitable $\gamma>0$. Thanks to Lemma \ref{lem: controllo rapporti}, it is necessary that $i=1$ and $j=2$. The knowledge of the limiting profile of the blow-down sequence $\{\tilde{\mf{v}}_n\}$ allows us to pass from the uniform convergence $\tilde{\mf{v}}_n \to \tilde{\mf{v}}$ to the $\mathcal{C}^{1,\alpha}$ convergence $\tilde v_{1,n} \to \tilde v_1$ far away from the free-boundary, and in particular in a set of type $\{\gamma x_1 > 2\delta\}$; here $\delta>0$ is a sufficiently small quantity, such that $\pa B_1 \cap \{\gamma x_1>2\delta\} \neq \emptyset$. To prove this, we observe that by uniform convergence $\tilde v_{1,n} \ge C>0$ in $B_2 \cap \{\gamma x_1>\delta\}$. Let $x_0 \in B_2 \cap \{\gamma x_1>2\delta\}$, and let $\rho >0$ so small that for any such $x_0$ the ball $B_\rho(x_0)$ is contained in $\{\gamma x_1>\delta\}$. Lemma \ref{lem: polynomial decay} applies to
\[
\begin{cases}
-\Delta \tilde v_{j,n} \le (\rho_n^2 r_n^2- C a_{1j} H(\mf{v}_n,0,\rho_n) \rho_n^2 M_n ) \tilde v_{j,n} \le - C a_{1j} H(\mf{v}_n,0,\rho_n) \rho_n^2 M_n \tilde v_{j,n} &  \\
\tilde v_{j,n} \ge 0 & \text{in $B_\rho(x_0)$} \\
\tilde v_{j,n} \le C,
\end{cases}
\]
for every $j\neq 1$, implying that
\[
H(\mf{v}_n,0,\rho_n) \rho_n^2 M_n \tilde v_{j,n}(x_0) \le C \qquad \text{for every $x_0 \in B_{2-\rho} \cap \{\gamma x_1>2\delta\}$}.
\]
As a consequence, $\{\Delta v_{1,n}\}$ is uniformly bounded in $B_{2-\rho} \cap \{\gamma x_1>2 \delta\}$, which together with the uniform bound of $\{v_{1,n}\}$ in the same set provides uniform boundedness in any space $\mathcal{C}^{1,\alpha}(B_{2-\rho} \cap \{\gamma x_1>2\delta\})$. Therefore the convergence $\tilde v_{1,n} \to \tilde v_1$ takes place in $\mathcal{C}^{1,\alpha}( B_{2-\rho} \cap \{\gamma x_1>2\delta\})$ for any $0<\alpha <1$.

This finally permits to reach a contradiction: indeed we have
\begin{align*}
0 &\ge \lim_{n \to \infty} \frac{\rho_n^2\int_{\pa B_{\rho_n}} |\nabla_\theta v_{1,n}|^2 + M_n a_{12} v_{1,n}^2 v_{2,n}^2-v_{1,n} f_{1,n}(x,v_{1,n})  }{\int_{\pa B_{\rho_n}} v_{1,n}^2} \\
& \ge \lim_{n \to \infty} \frac{\int_{\pa B_1} |\nabla_\theta \tilde v_{1,n}|^2}{\int_{\pa B_1} \tilde v_{1,n}^2} - \lim_{n \to \infty} \rho_n^2 \frac{\int_{\pa B_{\rho_n}} v_{1,n} f_{1,n}(x,v_{1,n}) }{\int_{\pa B_{\rho_n}} v_{1,n}^2} \ge C >0,
\end{align*}
where the last inequalities follow from the fact that by $\mathcal{C}^{1,\alpha}$ convergence
\[
\lim_{n \to \infty} \frac{\int_{\pa B_1} |\nabla_\theta \tilde v_{1,n}|^2}{\int_{\pa B_1} \tilde v_{1,n}^2} \ge \lim_{n \to \infty} \frac{\int_{\pa B_1 \cap \{\gamma x_1>2\delta\}} |\nabla_\theta \tilde v_{1,n}|^2}{\int_{\pa B_1} \tilde v_{1,n}^2}  = \frac{\int_{\pa B_1 \cap \{\gamma x_1>2\delta\}} |\nabla_\theta(\gamma x_1)|^2}{\int_{\pa B_1} (\gamma x_1^+)^2} =C>0,
\]
while on the contrary in light of the definition of $f_{i,n}$ and on assumption \eqref{be assumpt}
\[
\lim_{n \to \infty} \left|\rho_n^2 \frac{\int_{\pa B_{\rho_n}} v_{1,n} f_{1,n}(x,v_{1,n}) }{\int_{\pa B_{\rho_n}} v_{1,n}^2}\right| \le d \rho_n^2 r_n^2 = 0.
\]
At this point it remains to prove the result when $\rho_n \le \bar r$. In such a situation the proof is much easier, as it is not necessary to introduce the scaling $\{\tilde{ \mf{v}}_{n}\}$, but it is sufficient to argue on the original sequence $\{\mf{v}_n\}$.
If $M_n \to +\infty$, then since $N(\mf{v}_n,0,r) \le 1$ for every $r \le \bar r_n$ we have that up to a rotation $v_{1,n} \to v_1= \gamma x_1^+$, and the convergence takes place in $\mathcal{C}^{1,\alpha}$ in any compact subset of $\{ \gamma x_1>\delta\}$. Then the conclusion follows exactly as before. If $(M_n)$ is bounded, then as observed in Proposition \ref{prop: riassunto blow-up} $\mf{v}_n \to \mf{v}$ in $\mathcal{C}^{1,\alpha}_{\loc}(\R^N)$, and since both $v_1$ and $v_2$ are non-trivial (this follows from Lemma \ref{lem: 2 non deg}), by the strong maximum principle $v_1,v_2 >0$ in $\R^N$. Since $(\rho_n)$ is also bounded, up to a subsequence $\rho_n \to \bar \rho \ge 2$. Recalling that $M_n \ge C >0$, we deduce that
\[
\frac{\bar \rho^2\int_{B_{\bar \rho}} M_n a_{12} v_{1,n}^2 v_{2,n}^2 }{\int_{B_{\bar \rho}} v_{1,n}^2  } \ge C >0,
\]
which allows to obtain a contradiction.
\end{proof}

It remains to show that also $J_{1,n}(r)$ and $J_{2,n}(r)$ are positive in the whole range $[2,\bar r_n/3]$.

\begin{lemma}\label{lem: positivita 2}
There exists $C>0$ independent of $n$ such that
\[
J_{i,n}(r) > C \qquad \text{for every } r \in \left[2,\frac{\bar r_n}{3}\right],
\]
for $i=1,2$.
\end{lemma}
\begin{proof}
First of all, there exists $\bar C>0$ such that $J_{i,n}(r) \ge \bar C$ for every $r \in [2,10]$ and $i=1,2$. This is a simple consequence of the $\mathcal{C}(B_{10})$ and $H^1(B_{10})$ convergence $v_{i,n} \to v_i$, with $v_i \not \equiv 0$ and not constant in $B_{10}$ for $i=1,2$, and of the fact that $f_{i,n}(x,v_{i,n}(x)) \to 0$ uniformly in $B_{10}$, see Lemma \ref{lem: mega lemma}. Let now
\[
s_n:= \sup\left\{ s \in (2,\bar r_n/3): J_{i,n}(r) > \bar C/10\text{ for every $r \in (2,s)$}\right\}.
\]
Note that $s_n \ge 10$ is well defined. In light of Lemmas \ref{lem: 2 non deg} and  \ref{lem: controllo rapporti}-\ref{lem: lambda_i non zero}, the assumptions of Theorem \ref{thm: ACF} are satisfied in the interval $(2,s_n)$, uniformly in $n$. As a consequence there exists $C>0$ such that
\[
r \mapsto \frac{J_{1,n}(r) J_{2,n}(r)}{r^4} \exp\{-C M_n^{-1/4} r^{-1/2} + C r_n^2 r^2\}
\]
is monotone non-decreasing for $r \in (2,s_n)$. We claim that $s_n=\bar r_n/3$: indeed for every $r \in (2,s_n)$ it results
\[
J_n(r) \ge J_n(2) e^{- C M_n^{-1/4}+Cr_n^2} e^{C M_n^{-1/4} s_n^{-1/2} - C r_n^2 s_n^2} \ge \frac{\bar C}{10} > 0
\]
at least for $n$ sufficiently large, which proves the claim.
\end{proof}

\begin{proof}[Conclusion of the proof of Lemma \ref{lem: ACF uniform}]
By Lemmas \ref{lem: 2 non deg}-\ref{lem: positivita 2}, the assumptions of Theorem \ref{thm: ACF} are satisfied by $(v_{1,n},v_{2,n})$ for $r \in [2,\bar r_n/3]$, with constants $\mu, \lambda$ and $\eps$ independent of $n$.
\end{proof}


\section{The completely symmetric interaction}\label{sec: symmetric}

This section is devoted to the study of the Lipschitz uniform continuity of the system \eqref{lv syst}:
\[
\begin{cases}
        - \Delta u_{i,\beta} = f_{i,\beta}(x,u_{i,\beta}) - \beta u_{i,\beta} \tsum_{j\neq i} u_{j,\beta} \\
        u_{i,\beta} > 0.
\end{cases}
\]
In Section \ref{sec: asymptotic}, under rather general assumptions on the competition term, we established the asymptotic properties of two blow-up sequences $\{\mf{v}_n\}$ and $\{\bar{\mf{v}}_n\}$, which have been introduced starting from the assumption that a uniform bound on the Lipschitz norm of $\{\mf{u}_{\beta_n}\}$ does not exist.
In what follows, we will show the such asymptotic properties bring us to a contradiction. 
We shall make use of the celebrated almost monotonicity formula of Caffarelli-Jerison-Kenig \cite{cafjerken}, which we recall here in its original formulation. For any given $u, v \in H^1$ functions, we let
\[
    \Phi(r) := \left( \frac{1}{r^2} \int_{B_r} \frac{|\nabla u|^2}{|x|^{N-2}}\right) \left(\frac{1}{r^2}\int_{B_r} \frac{|\nabla v|^2}{|x|^{N-2}}\right).
\]
\begin{theorem}[Caffarelli-Jerison-Kenig almost monotonicity]
Suppose $u$, $v$ are non-negative, continuous functions on the unit ball $B_1$. Suppose that $-\Delta u \leq 1$ and $-\Delta v \leq 1$ in the sense of distributions and that $u(x) v(x) = 0$ for all $x \in B_1$. Then there exists a constant $C$ depending \emph{only} on the dimension such that
\[
    \Phi(r) \leq C \left(1 + \int_{B_r} \frac{|\nabla u|^2}{|x|^{N-2}} + \int_{B_r} \frac{|\nabla v|^2}{|x|^{N-2}} \right)^2, \qquad \text{for every $0 < r \leq 1$}.
\]
Moreover, if $u$ and $v$ satisfy the same assumptions also in the ball $B_2$, then there exists a dimensional constant $C>0$ such that
\[
    \Phi(r) \leq C\left( 1 + \int_{B_2} u^2 + \int_{B_2} v^2 \right)^2, \quad 0 < r \leq 1.
\]
\end{theorem}

One of the main consequences of the previous theorem is that the function $\Phi$ is bounded uniformly in $r$ whenever $u$ and $v$ are bounded in the ball $B_2$. In particular, in our setting we have the following.

\begin{lemma}
Let $\bar r > 0$ be fixed. There exists a constant $C>0$ independent of $\beta$ such that for any $r \leq \bar r/2$ and $x_0$ for which $B_{\bar r}(x_0) \subset \Omega$, the estimate
\begin{equation}\label{eqn: CJK estimate for LV}
    \frac{1}{r^N} \int_{B_r(x_0)} |\nabla \left( u_{i,\beta} - u_{j,\beta}\right)^+ |^2 \cdot \frac{1}{r^N} \int_{B_r(x_0)} |\nabla \left( u_{i,\beta} - u_{j,\beta}\right)^- |^2 \leq C
\end{equation}
holds for any $i \neq j$ and $\beta>0$.
\end{lemma}
\begin{proof}
For any two indices $i \neq j$, a straightforward computation shows that
\[
    - \Delta \left( u_{i,\beta} - u_{j,\beta}\right) + \beta \left( u_{i,\beta} - u_{j,\beta}\right) \sum_{h \neq i, j} u_{h,\beta} = f_{i,\beta}(x,u_{i,\beta}) - f_{j,\beta}(x,u_{j,\beta})
\]
where the right hand side is, by assumption, uniformly bounded in $L^\infty$. It follows that the positive and the negative part of $\left( u_{i,\beta} - u_{j,\beta}\right)$ fall under the assumptions of the Caffarelli-Jerison-Kenig monotonicity formula, and in particular
\begin{multline*}
    \frac{1}{r^N} \int_{B_r(x_0)} |\nabla \left( u_{i,\beta} - u_{j,\beta}\right)^+ |^2  \cdot \frac{1}{r^N} \int_{B_r(x_0)} |\nabla \left( u_{i,\beta} - u_{j,\beta}\right)^- |^2  \\
    \leq \left( \frac{1}{r^2} \int_{B_r(x_0)} \frac{|\nabla  \left( u_{i,\beta} - u_{j,\beta}\right)^+|^2}{|x-x_0|^{N-2}} \right) \left(\frac{1}{r^2}\int_{B_r(x_0)} \frac{|\nabla  \left( u_{i,\beta} - u_{j,\beta}\right)^-|^2}{|x-x_0|^{N-2}} \right) \\
    \leq C\|f_{i,\beta} - f_{j,\beta} \|_{L^{\infty}}^2 \left( 1+ \int_{B_2(x_0)} u_{i,\beta}^2 + \int_{B_2(x_0)} u_{j,\beta}^2 \right)^2
\end{multline*}
where the last term is, by assumption, uniformly bounded in $\beta$.
\end{proof}

\begin{corollary}\label{cor: ordering}
Any blow-up limit $\bv$ is made of ordered functions, that is, for any pair $i \neq j$ either $v_i \leq v_j$ or $v_j \leq v_i$ in the whole $\R^N$.
\end{corollary}
\begin{proof}
Indeed, scaling properly the estimate \eqref{eqn: CJK estimate for LV}, we obtain for every $r \in (0,1/r_n)$ and $n$ large enough
\[
    \frac{1}{r^N} \int_{B_r} |\nabla (v_{i,n} - v_{j,n})^+ |^2 \cdot \frac{1}{r^N} \int_{B_r} |\nabla (v_{j,n} - v_{i,n})^- |^2 \leq \frac{\eta(x_n)^4}{L_n^4} C \to 0
\]
as $n\to \infty$. The conclusion follows by strong $H^1_{\loc}(\R^N)$ convergence of the blow-up sequence and by the continuity of the blow-up limit.
\end{proof}

We now recall a classical result, for which we refer to Lemma 2 in \cite{Br84}.

\begin{lemma}\label{lem: liouville equation}
Let $1<p<\infty$, and let $u \in L^p_{\loc}(\R^N)$ be a solution of
\[
    - \Delta u \leq - |u|^{p-1}u \qquad \text{in $\R^N$},
\]
in the sense of distributions. Then $u \le 0$. In particular, if we assume $u$ to be non-negative, then $u \equiv 0$.
\end{lemma}

\begin{proof}[Proof of Theorem \ref{thm: main general lv}]
We divided the proof in two steps.

\paragraph{\textbf{Step 1)} \emph{the case $(M_n)$ bounded.}}
In this case by Proposition \ref{prop: riassunto blow-up} the limiting function $\bv$ is a non-negative, non-trivial, non-constant and sublinear solution of
\[
    -\Delta v_1 = -M_\infty v_1 v_2 \quad -\Delta v_2 = -M_\infty v_1 v_2 \quad v_j \equiv 0 \quad \text{for every $j \neq 1,2$},
\]
for some $M_\infty>0$. By Corollary \ref{cor: ordering} we evince that either $v_1 \ge v_2$ in $\R^N$, or $v_2 \ge v_1$ in $\R^N$. We shall show that in such situation $v_1 \equiv v_2 \equiv 0$ in $\R^N$, a contradiction. Without loss of generality, we suppose that $v_1 \not \equiv 0$ and $v_1 \ge v_2$. Thanks to Lemma \ref{lem: liouville equation}
\[
    - \Delta v_2 = -M_\infty v_1 v_2 \le -M_\infty v_2^2 \quad \implies \quad  v_2 \equiv 0.
\]
But then $v_1$ is a non-constant positive harmonic function, in contradiction with the classical Liouville theorem.


\paragraph{\textbf{Step 2)} \emph{the case $M_n \to +\infty$.}}
In such a situation, the segregation condition (Proposition \ref{prop: riassunto blow-up}, \eqref{eq segregata}) implies that all the component of the vector $\bv$ are all trivial with the exception of $v_1$, which is then a subharmonic non-constant function. Letting
\[
    \hat v_{1,n} := v_{1,n} - \sum_{j \neq 1} v_{j,n}
\]
we immediately obtain that also $\hat v_{1,n} \to v_1$ locally uniformly in $\R^N$. But at the same time, a direct computation shows that
\[
    - \Delta \hat v_{1,n} \geq f_{1,n}(x,v_{1,n}) - \sum_{j \neq 1} f_{j,n}(x,v_{j,n}) \qquad \text{in $\Omega_n$}
\]
where the right hand side vanishes uniformly as $n \to +\infty$, implying that the function $v_1$ is also superharmonic. This again forces $v_1$ to be a non-constant positive harmonic function, in contradiction with the classical Liouville theorem.

\end{proof}

\medskip

\noindent \textbf{Acknowledgements:} we wish to thank Benedetta Noris for a careful reading of the first version of this paper. The authors are partially supported through the project ERC Advanced Grant 2013 n. 339958 ``Complex Patterns for Strongly Interacting Dynamical Systems - COMPAT''.


\begin{thebibliography}{10}

\bibitem{acf}
H.~W. Alt, L.~A. Caffarelli, and A.~Friedman.
\newblock Variational problems with two phases and their free boundaries.
\newblock {\em Trans. Amer. Math. Soc.}, 282(2):431--461, 1984.

\bibitem{blwz_phase}
H.~Berestycki, T.-C. Lin, J.~Wei, and C.~Zhao.
\newblock On phase-separation models: asymptotics and qualitative properties.
\newblock {\em Arch. Ration. Mech. Anal.}, 208(1):163--200, 2013.

\bibitem{BeTeWaWe}
H.~Berestycki, S.~Terracini, K.~Wang, and J.~Wei.
\newblock On entire solutions of an elliptic system modeling phase separations.
\newblock {\em Adv. Math.}, 243:102--126, 2013.

\bibitem{Br84}
H.~Brezis.
\newblock Semilinear equations in {${\bf R}^N$} without condition at infinity.
\newblock {\em Appl. Math. Optim.}, 12(3):271--282, 1984.

\bibitem{cafjerken}
L.~A. Caffarelli, D.~Jerison, and C.~E. Kenig.
\newblock Some new monotonicity theorems with applications to free boundary
  problems.
\newblock {\em Ann. of Math. (2)}, 155(2):369--404, 2002.

\bibitem{ckl}
L.~A. Caffarelli, A.~L. Karakhanyan, and F.-H. Lin.
\newblock The geometry of solutions to a segregation problem for nondivergence
  systems.
\newblock {\em J. Fixed Point Theory Appl.}, 5(2):319--351, 2009.

\bibitem{caflin}
L.~A. Caffarelli and F.-H. Lin.
\newblock Singularly perturbed elliptic systems and multi-valued harmonic
  functions with free boundaries.
\newblock {\em J. Amer. Math. Soc.}, 21(3):847--862, 2008.

\bibitem{cafroq}
L.~A. Caffarelli and J.-M. Roquejoffre.
\newblock Uniform {H}\"older estimates in a class of elliptic systems and
  applications to singular limits in models for diffusion flames.
\newblock {\em Arch. Ration. Mech. Anal.}, 183(3):457--487, 2007.

\bibitem{csbook}
L.~A. Caffarelli and S.~Salsa.
\newblock {\em A geometric approach to free boundary problems}, volume~68 of
  {\em Graduate Studies in Mathematics}.
\newblock American Mathematical Society, Providence, RI, 2005.

\bibitem{clll}
S.-M. Chang, C.-S. Lin, T.-C. Lin, and W.-W. Lin.
\newblock Segregated nodal domains of two-dimensional multispecies
  {B}ose-{E}instein condensates.
\newblock {\em Phys. D}, 196(3-4):341--361, 2004.


\bibitem{ctvIndiana}
M.~Conti, S.~Terracini, and G.~Verzini.
\newblock A variational problem for the spacial segregation of reaction-diffusion systems.
\newblock {\em Indiana Univ. Math. J.}, 54(3):779--815, 2005.


\bibitem{ctvNehari}
M.~Conti, S.~Terracini, and G.~Verzini.
\newblock Nehari's problem and competing species systems.
\newblock {\em Ann. Inst. H. Poincar\'e Anal. Non Lin\'eaire}, 19(6):871--888,
  2002.

\bibitem{ctvOptimal}
M.~Conti, S.~Terracini, and G.~Verzini.
\newblock An optimal partition problem related to nonlinear eigenvalues.
\newblock {\em J. Funct. Anal.}, 198(1):160--196, 2003.

\bibitem{ctv}
M.~Conti, S.~Terracini, and G.~Verzini.
\newblock Asymptotic estimates for the spatial segregation of competitive
  systems.
\newblock {\em Adv. Math.}, 195(2):524--560, 2005.

\bibitem{dwz1}
E.~N. Dancer, K.~Wang, and Z.~Zhang.
\newblock The limiting equation for the Gross-Pitaevskii equations and S. Terracini's conjectures.
\newblock {\em J. Funct. Anal.}, 262(3):1087--1131, 2012.


\bibitem{DaWaZh2}
E.~N. Dancer, K.~Wang, and Z.~Zhang.
\newblock Uniform {H}\"older estimate for singularly perturbed parabolic
  systems of {B}ose-{E}instein condensates and competing species.
\newblock {\em J. Differential Equations}, 251(10):2737--2769, 2011.

\bibitem{FaSo}
A.~Farina and N.~Soave.
\newblock Monotonicity and 1-dimensional symmetry for solutions of an elliptic
  system arising in {B}ose-{E}instein condensation.
\newblock Archives Ration. Mech. Anal. \textbf{213} (1) (2014), 287--326.



\bibitem{GiTr}
D.~Gilbarg and N.~S. Trudinger.
\newblock {\em Elliptic partial differential equations of second order}, volume
  224 of {\em Grundlehren der Mathematischen Wissenschaften [Fundamental
  Principles of Mathematical Sciences]}.
\newblock Springer-Verlag, Berlin, second edition, 1983.

\bibitem{matpet}
N.~Matevosyan and A.~Petrosyan.
\newblock Almost monotonicity formulas for elliptic and parabolic operators
  with variable coefficients.
\newblock {\em Comm. Pure Appl. Math.}, 64(2):271--311, 2011.

\bibitem{Mi}
M.~Mimura.
\newblock Spatial distributions of competiing species, Mathematical Ecology
  (Trieste, 1982), Lecture Notes in Biomathematics, vol. 54, 1982, 492–-501.

\bibitem{nttv}
B.~Noris, H.~Tavares, S.~Terracini, and G.~Verzini.
\newblock Uniform {H}\"older bounds for nonlinear {S}chr\"odinger systems with
  strong competition.
\newblock {\em Comm. Pure Appl. Math.}, 63(3):267--302, 2010.

\bibitem{PaWa}
A.~S. Parkins and D.~F. Walls.
\newblock The physics of trapped dilute-gas {B}ose-{E}instein condensates.
\newblock {\em Phys. Rep.}, 303:1--80, 1998.

\bibitem{PiSt}
L.~Pitaevskii and S.~Stringari.
\newblock {\em {B}ose-–{E}instein condensation}.
\newblock Oxford, 2003.

\bibitem{qui}
V.~Quitalo.
\newblock A free boundary problem arising from segregation of populations with
  high competition.
\newblock {\em Arch. Ration. Mech. Anal.}, 210(3):857--908, 2013.

\bibitem{rtt}
M.~Ramos, H.~Tavares, and S.~Terracini.
\newblock Existence and regularity of solutions to optimal partition problems
  involving laplacian eigenvalues.
\newblock Preprint arXiv:1403.6313.

\bibitem{RuCaFu}
C.~R\"uegg, N.~Cavadini, A.~Furrer, H.~U. G\"udel, K.~Kr\"amer, H.~Mutka,
  A.~Wildes, K.~Habicht, and P.~Vorderwisch.
\newblock {B}ose–-{E}instein condensation of the triplet states in the
  magnetic insulator {T}l{C}u{C}l3.
\newblock {\em Nature}, 423:62--65, 2003.

\bibitem{ShKaTe}
N.~Shigesada, K.~Kawasaki, and E.~Teramoto.
\newblock The effects of interference competition on stability, structure and
  invasion of a multi species system, J. Math. Biol. 21(2)(1984) 97–-113.

\bibitem{stz}
N.~Soave, H.~Tavares, and A.~Zilio.
\newblock Quasi-optimal local bounds for strongly competing Schr\"odinger
  equations with nontrivial grouping (preprint).

\bibitem{SoTe}
N.~Soave and S.~Terracini.
\newblock Liouville theorems and $1$-dimensional symmetry for solutions of an elliptic system modelling phase separation.
\newblock {\em Advances in Math.}, to appear. Preprint arXiv:1404.7288.

\bibitem{SoZi}
\newblock N. Soave and A. Zilio,
\newblock Entire solutions with exponential growth for an elliptic system modeling phase separation.
\newblock Nonlinearity \textbf{27} (2) (2014), 305--342.


\bibitem{tt}
H.~Tavares and S.~Terracini.
\newblock Regularity of the nodal set of segregated critical configurations
  under a weak reflection law.
\newblock {\em Calc. Var. Partial Differential Equations}, 45(3-4):273--317,
  2012.

\bibitem{tvz1}
S.~Terracini, G.~Verzini, and A.~Zilio.
\newblock Uniform {H}\"older bounds for strongly competing systems involving
  the square root of the laplacian.
\newblock Preprint arXiv:1211.6087.

\bibitem{tvz2}
S.~Terracini, G.~Verzini, and A.~Zilio.
\newblock Uniform {H}\"older regularity with small exponent in
  competition-fractional diffusion systems.
\newblock {\em Discrete Contin. Dyn. Syst.}, 34(6):2669--2691,
  2014.

\bibitem{Timm}
E.~Timmermans.
\newblock Phase separation of {B}ose-{E}instein condensates.
\newblock {\em Phys. Rev. Lett.}, 81:5718--5721, 1998.

\bibitem{vz}
G.~Verzini and A.~Zilio.
\newblock Strong competition versus fractional diffusion: the case of
  {L}otka-{V}olterra interaction.
\newblock {\em Comm. Partial Differential Equations}, in press doi:10.1080/03605302.2014.890627.

\bibitem{Wa}
K.~Wang.
\newblock On the {D}e {G}iorgi type conjecture for an elliptic system modeling
  phase separation.
\newblock {\em Comm. Partial Differential Equations}, 39(4):696--739, 2014.

\bibitem{WaZh}
K.~Wang and Z.~Zhang.
\newblock Some new results in competing systems with many species.
\newblock {\em Ann. Inst. H. Poincar\'e Anal. Non Lin\'eaire}, 27(2):739--761, 2010.



\bibitem{ww}
J.~Wei and T.~Weth.
\newblock Asymptotic behaviour of solutions of planar elliptic systems with
  strong competition.
\newblock {\em Nonlinearity}, 21(2):305--317, 2008.

\end{thebibliography}

\end{document}